\documentclass[11pt,notitlepage,twoside,a4paper]{amsart}
\usepackage{mathrsfs}

\usepackage{amsmath,amssymb,enumerate}

\usepackage{epsfig,fancyhdr,color}%,showkeys,amsmidx

\usepackage{CJK}
\usepackage{epstopdf}
\usepackage{amssymb}
\usepackage{amsmath,amsthm}
\usepackage{latexsym}
\usepackage{amscd}
\usepackage{psfrag}
\usepackage{graphicx}
\usepackage{epsf}
\usepackage[latin1]{inputenc}
\usepackage[all]{xy}
\usepackage{tikz}
\usepackage{prettyref}
\usepackage{subfigure}
\usepackage{float}

\newtheorem{theorem}{\rm\bf Theorem}[section]

\newtheorem{proposition}[theorem]{\rm\bf Proposition}
\newtheorem{lemma}[theorem]{\rm\bf Lemma}
\newtheorem{corollary}[theorem]{\rm\bf Corollary}

\newtheorem*{proposition 1}{\rm\bf Proposition 1}
\newtheorem*{proposition 2}{\rm\bf Proposition 2}

\theoremstyle{definition}
\newtheorem{definition}[theorem]{\rm\bf Definition}

\theoremstyle{remark}
\newtheorem{remark}[theorem]{\rm\bf Remark}

\newtheorem{example}[theorem]{\rm\bf Example}

\newtheorem{problem}[theorem]{\rm\bf Problem}

\newtheorem*{case 1}{\rm\bf Case 1}
\newtheorem*{case 2}{\rm\bf Case 2}
\newtheorem*{case 3}{\rm\bf Case 3}
\newtheorem*{case 4}{\rm\bf Case 4}

\newtheorem*{subcase 1}{\rm\bf Subcase 1}
\newtheorem*{subcase 2}{\rm\bf Subcase 2}
\newtheorem*{subcase 3}{\rm\bf Subcase 3}
\newtheorem*{subcase 4}{\rm\bf Subcase 4}

\def\interieur#1{\mathord{\mathop{\kern 0pt #1}\limits^\circ}}

\newcounter{ClaimProofCounter}
\newcounter{SthProofCounter}
\newcounter{CaseCounter}
    \newenvironment{case}[1]
	{
        \refstepcounter{CaseCounter}%
        \leftline{\large{\textsl{Case #1}}}
	}
	{
	}
\newcounter{SubcaseCounter}
\title[On metrics defined by length spectrum]{On metrics defined by length spectra on Teichm\"uller spaces of surfaces with boundary}

\author{Youliang Zhong}
\address{Youliang Zhong: Department of Mathematics, Sun Yat-sen University, 510275, Guangzhou, P. R. China}
\email{zhongyl0430@gmail.com}

\author{Lixin Liu}
\address{Lixin Liu: Department of Mathematics, Sun Yat-sen University, 510275, Guangzhou, P. R. China}
\email{mcsllx@mail.sysu.edu.cn}

\author{Weixu Su}
\address{Weixu Su: Department of Mathematics, Fudan University, 200433, Shanghai, P. R. China}
\email{suwx@fudan.edu.cn}

\thanks{L. Liu and Y. Zhong are partially supported by NSFC No: 11271378; W. Su is partially supported by NSFC No: 11201078.}

\date{\today}

\begin{document}

	\begin{abstract}
		We prove that the length spectrum metric and the arc-length spectrum metric are almost-isometric
        on the $\epsilon_0$-relative part of Teichm\"uller spaces of surfaces with boundary.
	\bigskip

	
	\end{abstract}

	\maketitle
	
	\section{Introduction}\label{intro}
	
		Let $S = S_{g,p,b}$ be a connected oriented surface of genus $g\geq 0$
 with $p \geq 0$  punctures and $b \geq 0$  boundary components.
		The boundary of $S$ is denoted by $\partial S$.
		The Euler characteristic of $S$ is $\mathcal{X}(S) = 2-2g-p-b$.
		In this paper, we always assume that $g \geq 0$, $p \geq 0$, $b \geq 1$ and $\mathcal{X}(S) < 0$.

In the following, all hyperbolic metrics on $S$ are assumed to be complete and totally geodesic on the
 boundary components.
 By the assumption that $\mathcal{X}(S) < 0$, there always exist a hyperbolic metric on $S$.

		A \emph{marked hyperbolic metric $(X,f)$} is a hyperbolic metric $X$ on $S$ equipped with
an orientation-preserving homeomorphism $f: S \to X$, where $f$ maps each component of $\partial S$ to a geodesic boundary
of $X$ and maps punctures to cusps.  The \emph{reduced Teichm\"uller space} of $S$, denoted by
$\mathcal{T}(S)$, is the set of equivalence classes of marked hyperbolic metrics on $S$, where two markings
		$(X_1,f_1)$ and $(X_2,f_2)$ are  \emph{equivalent}
        if there is an isometry  $h : X_1 \to X_2$  homotopic to $f_2 \circ \ f_1^{-1}$.
		We should point out that, in this reduced theory, homotopies do not necessarily fix  $\partial S$ pointwise.
		The notion of a reduced Teichm\"uller space was introduced by Earle \cite{Earle, ES}, where he defined the space by using
quasiconformal deformations of Fuchsian groups (of the second kind).

 Since all Teichm\"uller spaces that we consider are reduced, we shall omit the word ``reduced" in this paper.
		For the sake of simplicity, we shall denote a marked hyperbolic surface $(X,f)$ or its equivalence class in $\mathcal{T}(S)$ by $X$, without explicit reference to the marking or to the equivalence relation.

	There are several natural metrics on Teichm\"uller space, e.g. the classical Teichm\"uller metric and the Weil-Petersson
metric. In this paper, we will study the length spectrum metric and the arc-length spectrum metric.
The length spectrum  metric was first studied by Sorvali \cite{S,Sorvali}, which can also be considered as the symmetrization
of an asymmetric Finsler metric defined by Thurston \cite{Thurston1998}.
The arc-length spectrum metric is new, which is defined only on
Teichm\"uller spaces of surfaces with boundary.
Both of the two above metrics are defined by using hyperbolic (or geodesic) length functions.
 There is no doubt that hyperbolic length is one of the most fundamental tools in Teichm\"uller theory.
We note that by recent works of Danciger, Gu\'eritaud and Kassel  \cite{DGK},
deformations of hyperbolic surfaces with boundary is related to Margulis spacetimes in Lorentz geometry.

\subsection{Metrics defined by length spectra.}
 To provide concrete definitions and state our results, we fix some terminology and
notation.
		
A simple closed curve on $S$ is said to be \emph{peripheral} if it is isotopic to a puncture.
		It is said to be \emph{essential} if it is neither peripheral nor isotopic to a point.
		It should be noticed that an essential closed curve may be isotopic to a boundary component.
		We denote by $\mathcal{C}(S)$ the set of homotopy classes of essential simple closed curves on $S$.
		
        An \textit{arc} on $S$ is the homeomorphic image of a closed interval which is properly embedded in $S$,
that is, the interior of the arc is in the interior of $S$ and the endpoints
of the arc lie on  $\partial S$. An arc is said to be \textit{essential} if it is not isotopic to a subset of $\partial{S}$.
		All homotopies of arcs that we consider here are relative to $\partial{S}$.
However, we don't require homotopies to fix  $\partial S$ pointwise.
		Let $\mathcal{B}(S)$ be the set of homotopy classes of essential arcs on $S$.
		
		For any $\alpha \in \mathcal{B}(S) \cup \mathcal{C}(S)$ and $X \in \mathcal{T}(S)$, we denote by $\ell_{X}(\alpha)$ the \emph{hyperbolic length} of $\alpha$, that is, the length of the geodesic representation of $\alpha$ under the hyperbolic metric $X$.

		For surfaces without boundary,  Thurston \cite{Thurston1998} defined the following asymmetric metric:

		\begin{equation*}
			\begin{split}
				d(X,Y) = \log \sup_{\alpha \in  \mathcal{C}(S)} \cfrac{\ell_Y(\alpha)}{\ell_X(\alpha)}.
			\end{split}
		\end{equation*}

For surfaces with boundary, the following asymmetric metric is a natural generalization of Thurston's formula  \cite{LPST, ALPS2} :
		\begin{equation*}
			\begin{split}
				\bar{d}(X,Y) = \log \sup_{\alpha \in \mathcal{C}(S) \cup \mathcal{B}(S)} \cfrac{\ell_Y(\alpha)}{\ell_X(\alpha)}.
			\end{split}
		\end{equation*}
Both of the above two metrics satisfy the separation axiom and triangle inequality, but none of them satisfies the symmetric condition.  	

\begin{remark}
For surfaces with boundary,  there exist (see \cite{Papado-Th2010}) distinct hyperbolic structures $X$ and $Y$ on $S$ such that for any element $\alpha\in \mathcal{C}(S)$,
$\frac{l_X (\alpha)} {l_Y(\alpha)} <1$. This implies that
  $$\log \sup_{\alpha \in \mathcal{C}(S)} \cfrac{\ell_X(\alpha)}{\ell_Y(\alpha)} \le 0.$$
As a result,  it's necessary to consider the union of closed curves and arcs in the definition of $\bar d$.
\end{remark}

\begin{definition}
 The \textit{length spectrum metric} $d_L$ on $\mathcal{T}(S)$ is defined by
		\begin{equation*}
			\begin{split}
				d_L(X,Y)
= \log \sup_{\alpha \in \mathcal{C}(S) } \{ \cfrac{\ell_X(\alpha)}{\ell_Y(\alpha)}, \   \cfrac{\ell_Y(\alpha)}{\ell_X(\alpha)} \}.
			\end{split}
		\end{equation*}
\end{definition}
\begin{definition}
The \textit{arc-length spectrum metric} $\delta_L$ on $\mathcal{T}(S)$ is defined by
		\begin{equation*}
			\begin{split}
				\delta_L (X,Y)
                &= \max \{ \bar d(X,Y), \ \bar{d}(Y,X) \} \\
%&= \log \max \{ \sup_{\alpha \in \mathcal{C}(S) \bigcup \mathcal{B}(S)} \cfrac{\ell_X(\alpha)}{\ell_Y(\alpha)}, \  \sup_{\alpha \in \mathcal{C}(S) \bigcup \mathcal{B}(S)} \cfrac{\ell_Y(\alpha)}{\ell_X(\alpha)} \} \\
				&= \log \sup_{\alpha \in \mathcal{C}(S) \cup \mathcal{B}(S)} \{ \cfrac{\ell_X(\alpha)}{\ell_Y(\alpha)}, \   \cfrac{\ell_Y(\alpha)}{\ell_X(\alpha)} \}.
			\end{split}
		\end{equation*}
\end{definition}
        The fact that $d_L$ is a  metric on $\mathcal{T}(S)$ was proved in \cite{S, Liu1999}.
        It is obvious that $d_L\leq \delta_L$.
        When $b=0$, since $\mathcal{B}(S)$ is empty, $d_L=\delta_L$.
For more works about the length spectrum metric, one refers to \cite{Rafi,Li,LSW,Liu1999,Liu2001,AlmostModuli,LPST,Papado-Th2007, Shiga}.	

\subsection{Main theorems.}
The aim of this paper is to compare the length spectrum metric with the arc-length spectrum on a large subset of $\mathcal{T}(S)$.

	\begin{definition}	
        Given $\epsilon_0 > 0$, the \emph{$\epsilon_0$-relative part} of $\mathcal{T}(S)$ is the subset of $\mathcal{T}(S)$ consisting of hyperbolic metrics
        with lengths of all boundary components bounded above by $\epsilon_0$.
        \end{definition}

In this paper we prove:

\begin{theorem}\label{lemma:first}
			
			There is a constant  $C$  depending on $\epsilon_0$ such that
			$$ d_{L}(X,Y) \leq \delta_L (X,Y)\leq d_{L }(X,Y)+ C$$
 for any $X$, $Y$ in the $\epsilon_0$-relative part of $\mathcal{T}(S)$.
			
		\end{theorem}
The left-hand side inequality follows by definition. The right-hand side inequality is equivalent to the following result:
		\begin{theorem}\label{lemma:key}
			There exists a positive constant $K$ depending on $\epsilon_0$ such that
			\begin{equation*}
				\begin{split}
					\sup_{\beta \in \mathcal{C}(S) \bigcup \mathcal{B}(S)} \{ \cfrac{\ell_{X_1}(\beta)}{\ell_{X_2}(\beta)}, \cfrac{\ell_{X_2}(\beta)}{\ell_{X_1}(\beta)} \}
					&\leq K \cdot \sup_{\alpha \in \mathcal{C}(S)} \{ \cfrac{\ell_{X_1}(\alpha)}{\ell_{X_2}(\alpha)}, \cfrac{\ell_{X_2}(\alpha)}{\ell_{X_1}(\alpha)} \}
				\end{split}
			\end{equation*}	
 for any $X_1$, $X_2$ in the $\epsilon_0$-relative part of $\mathcal{T}(S)$.		
		\end{theorem}

\begin{remark}
Recall that  a map $f: M\rightarrow N$ between metric spaces is
called a \emph{$(\lambda,C)$ quasi-isometry} (with given constants $C\geq 0$ and $\lambda\geq 1$)
if
$$\frac{1}{\lambda}d_M(x,y) - C\leq d_N(f(x),f(y))\leq \lambda d_M(x,y) + C$$
for all $x,y\in M$, and the $C$-neighborhood of $f(M)$ in $N$ is all of $N$.  An $(1,C)$ quasi-isometry is called an \emph{almost-isometry}.

 Theorem \ref{lemma:first} implies that the length spectrum metric and the arc-length spectrum metric are almost-isometric on the $\epsilon_0$-relative part of $\mathcal{T}(S)$.
 \end{remark}

 For  $0 < \epsilon < \epsilon_0$, the \emph{$\epsilon$-thick part} of $\mathcal{T}(S)$ is the subset of $\mathcal{T}(S)$ consisting of hyperbolic metrics $X$ with hyperbolic length $\ell_X(\alpha)$  not less than $\epsilon$ for all $\alpha \in \mathcal{C}(S)$.
		The intersection of the $\epsilon$-thick part and the $\epsilon_0$-relative part of $\mathcal{T}(S)$ is called the \emph{$\epsilon_0$-relative $\epsilon$-thick part} of $\mathcal{T}(S)$.
  We can deduce from \cite[Theorem 3.6]{LPST} that the length spectrum metric and the arc-length spectrum metric are almost-isometric  on the $\epsilon_0$-relative $\epsilon$-thick part  of $\mathcal{T}(S)$.
In fact, by \cite[Proposition 3.5]{LPST}, there exists a positive constant $K_0$ depending on $\epsilon$ and $\epsilon_0$ such that, for any $X_1$, $X_2$ in the $\epsilon$-thick $\epsilon_0$-relative part of $\mathcal{T}(S)$,

\begin{equation}\label{eq:couter}
\sup_{\beta \in \mathcal{C}(S) \bigcup \mathcal{B}(S)} \{ \cfrac{\ell_{X_2}(\beta)}{\ell_{X_1}(\beta)} \}
					\leq K_0 \cdot \sup_{\alpha \in \mathcal{C}(S)} \{ \cfrac{\ell_{X_2}(\alpha)}{\ell_{X_1}(\alpha)} \}.
\end{equation}
			However, the above inequality does not hold on the whole $\epsilon_0$-relative part of $\mathcal{T}(S)$.
A counter example is constructed at the end of Section \ref{ProofTheorem2}
(Example \ref{example:Nielsen}).
As a result, Theorem \ref{lemma:first}  can be seen as an extension of \cite[Theorem 3.6]{LPST}.

\begin{remark} We should mention that, in the statement of \cite[Proposition 3.5]{LPST} the constant $K_0$ depends on $\epsilon$, $\epsilon_0$ and the topology of $S$. But during the  proof of \eqref{eq:couter}, the constant $K_0$ only depend on $\epsilon$ and $\epsilon_0$. Similarly, the constants $C$ and $K$ in Theorem \ref{lemma:first} and Theorem \ref{lemma:key}  are independent of
the topology of the surface $S$.
\end{remark}

       \subsection{Outline of the paper.} In Section 2 we will recall some elementary results in hyperbolic geometry
        that we need later. The proof of Theorem \ref{lemma:key} will be given in Section 3 and Section 4. In Section 3, we
       deal with the case where the constant $\epsilon_0$ is sufficiently small.
       In Section 4, we use the results in Section 3 to prove Theorem \ref{lemma:key} in the general case.

        To prove Theorem \ref{lemma:key}, we will use the technique of ``replacing an arc by a loop"
        to show that the length ratio of  an arc can be controlled by the length ratio of some appropriated simple closed curve.
        Such an idea was initiated by Minsky \cite{Minsky} and it has many applications (see, e.g., Rafi \cite{Rafi}).

        We will discuss related results on moduli spaces and on
         surfaces of infinite type in Section 5.

 \subsection{Acknowledgements}
 The authors would like to thank
the referee for many corrections and useful suggestions.
	
	\section{Preliminaries}\label{Preliminaries}

		\subsection{Formulae for right-angle pentagon and hexagon.}
		For a right-angled pentagon on the hyperbolic plane with consecutive side lengths $a$, $b$, $\alpha$, $c$ and $\beta$, as in Figure $1$, we have
		\begin{equation}\label{eq:RightAngledPentagon}
			\cosh c = \sinh a \ \sinh b.
		\end{equation}
		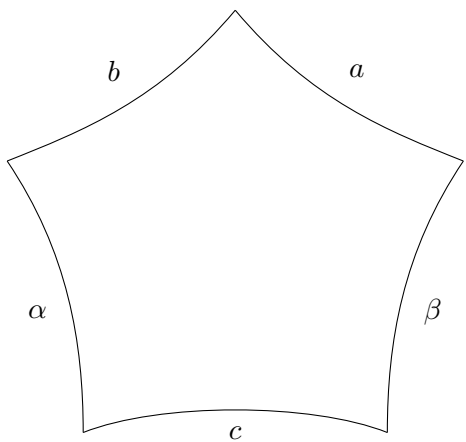
\begin{figure}[ht]\label{fig:pentagon}
			\begin{tikzpicture}[scale=0.4]

				\draw		(0,0)	node	{$c$};
				\draw		(-5,0)	..controls	(-2.5,1)	and	(2.5,1)..	(5,0);
				\draw		(6.5,4)	node	{$\beta$};
				\draw		(5,0)	..controls	(5,3)	and	(5.5,6)..	(7.5,9);
				\draw		(-6.5,4)	node	{$\alpha$};
				\draw		(-5,0)	..controls	(-5,3)	and	(-5.5,6)..	(-7.5,9);
				\draw		(4,12)	node	{$a$};
				\draw		(0,14)	..controls	(2.5,11)	and	(5,10)..	(7.5,9);
				\draw		(-4,12)	node	{$b$};
				\draw		(0,14)	..controls	(-2.5,11)	and	(-5,10)..	(-7.5,9);

			\end{tikzpicture}
			\caption{\small{An example of pentagon.}}
		\end{figure}

		We also need the following formula for a right-angled hexagon with consecutive side lengths $a$, $\gamma$, $b$, $\alpha$, $c$ and $\beta$,
 as in Figure $2$ :
		\begin{equation}\label{eq:RightAngledHexagon}
			\cosh c + \cosh a \cosh b = \sinh a \sinh b \cosh \gamma.
		\end{equation}
		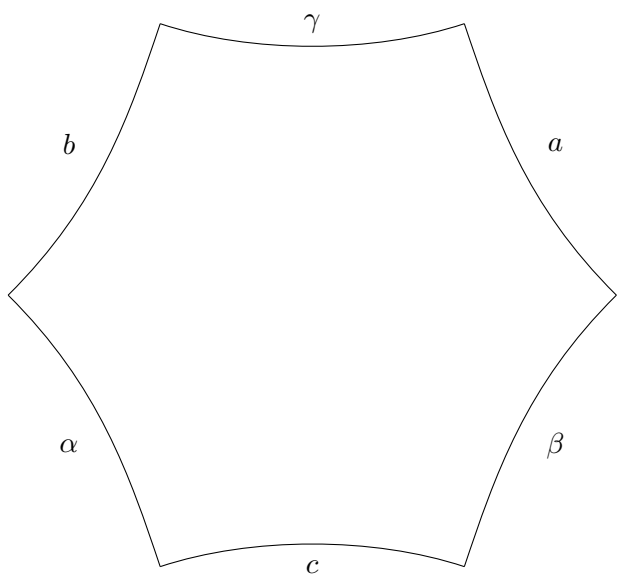
\begin{figure}[ht]\label{fig:Hexagon}
			\begin{tikzpicture}[scale=0.4]
				\draw		(8,5)	node	{$a$};
				\draw		(10,0)	..controls	(7,3)	and	(6,6)..	(5,9);
				\draw		(-8,5)	node	{$b$};
				\draw		(-10,0)	..controls	(-7,3)	and	(-6,6)..	(-5,9);
				\draw		(0,9)	node	{$\gamma$};
				\draw		(5,9)	..controls	(2,8)	and	(-2,8)..	(-5,9);

				\draw		(8,-5)	node	{$\beta$};
				\draw		(10,0)	..controls	(7,-3)	and	(6,-6)..	(5,-9);
				\draw		(-8,-5)	node	{$\alpha$};
				\draw		(-10,0)	..controls	(-7,-3)	and	(-6,-6)..	(-5,-9);
				\draw		(0,-9)	node	{$c$};
				\draw		(5,-9)	..controls	(2,-8)	and	(-2,-8)..	(-5,-9);

			\end{tikzpicture}
			\caption{\small{An example of hexagon. }}
		\end{figure}

		The inverse hyperbolic sine function and the inverse hyperbolic cosine function are given by
		\begin{equation}\label{eq:Arsinh}
			{\sinh^{-1}} \ x = \ln (x + \sqrt{x^2+1})
		\end{equation}
		and
		\begin{equation}\label{eq:Arcosh}
			{\cosh^{-1}} \ x = \ln (x + \sqrt{x^2-1}) \ ,\text{ for } x \geq 1.
		\end{equation}

		Let $f$ and $g$ be any two functions defined on  a set $U$. We call $f \asymp g$ if
 there exists a positive constant $C$ such that
		\begin{equation*}
				C^{-1} \cdot f(\tau) \leq g(\tau) \leq C \cdot f(\tau), \ \forall \ \tau \in U .
		\end{equation*}
		Usually the constant $C$ will depend on the choice of $U$.

Given $\epsilon_0>0$, we have
  $x \asymp \sinh x$ if $x \leq \epsilon_0$, and
 $\sinh x \asymp e^{x}$ if $x \geq \epsilon_0$.
Here it is obvious that the multiplicative constants for $\asymp$ depend on $\epsilon_0$.

        \subsection{Regular annulus.}
		 Let $X$ be a hyperbolic structure on $S$ and denote the distance between two distinct points $p$ and $q$ on $X$ by $d_X(p,q)$.
        The distance between two subsets $S_1$ and $S_2$ of $X$ is defined by
        $$d_X(S_1,S_2)=\inf_{x_1 \in S_1, x_2 \in S_2} d_X(x_1,x_2).$$

        Let $A$ be an annulus embedded in $S$.
        Denote the two boundaries of $A$ by $\gamma$ and $\gamma'$.
The annulus $A$ is said to be \textit{regular}  if there is a constant $w>0$ such that
         $$d_X(p,\gamma') = d_X(p',\gamma) = w, \forall \ p \in \gamma, p' \in \gamma'.$$

		For a positive number $\delta$ and a simple closed geodesic $\gamma$ on $X$
        (either in the interior of $X$ or be a boundary component of $X$),
        we denote the $\delta$-neighborhood of $\gamma$ by
        $$A_{\delta}(\gamma)=\{x \in X \ | \ d_X(x,\gamma) < \delta\}.$$
By the Collar Lemma (ref. \cite{Buser}), $A_{\delta}(\gamma)$ is a regular annulus
         if $$\delta\leq \sinh^{-1} (\frac{1}{\sinh \frac{\ell_X(\gamma)}{2}}).$$

        Suppose that $A_{\delta}(\gamma)$ is a regular annulus. If $\gamma$ is in the interior of $X$, then the width of $A_{\delta}(\gamma)$ is equal to $2\delta$ and $\gamma$ lies in the middle of
         $A_{\delta}(\gamma)$.  If $\gamma$ is a boundary component of $S$, then the width of $A_{\delta}(\gamma)$ is equal to $\delta$.
  In both cases, we will say that $A_{\delta}(\gamma)$ is a \emph{regular annulus around $\gamma$}.

  We define the auxiliary function  $\eta(x)$ by
        \begin{equation*}
            \eta(x) := \sinh^{-1} (\frac{1}{\sinh \frac{x}{2}})=\cfrac{1}{2} \ln \cfrac{\cosh(x/2)+1}{\cosh(x/2)-1} \ .
        \end{equation*}

        By the Collar Lemma again,
        for any two distinct simple closed geodesics $\gamma_1$ and $\gamma_2$ on $X$,
        the regular annuli $A_{\eta(\ell_{X}(\gamma_1))}(\gamma_1)$  and $A_{\eta(\ell_{X}(\gamma_2))}(\gamma_2)$ are disjoint.

 Throughout this paper, we only consider  regular annuli as collar neighborhoods of boundary components of $X$.
		In this case, as we show in Figure 3, the geodesic $\gamma$ is a  boundary component of $X$.
The regular annulus contains $\gamma$ and $\gamma'$ as boundary components. We call $\gamma'$ the \emph{inner boundary} of
$A_{\delta}(\gamma)$
and denote the length of $\gamma'$ on $X$ by $\ell_X(\gamma')$ (even through $\gamma'$ is not a geodesic, now and later, we will
use the notation $\ell$ to denote the length of an inner boundary when there is no cause of confusion).
The relation between $\ell_X(\gamma)$ and $\ell_X(\gamma')$ is given by (see \cite{Rafi, Marden})
		\begin{equation}\label{eq:RegularAnnulis}
			\ell_X(\gamma')=\ell_X(\gamma) \cdot \cosh d_X(\gamma, \gamma').
		\end{equation}

		\begin{figure}[ht]\label{fig:RegularAnnular}
			\begin{tikzpicture}[scale=0.4]
				\draw		(-1.5,0)	node	{$\gamma$};
				\draw		(23,0)	node	{$\gamma'$};
				\draw		(12,1)	node	{$d_X(\gamma, \gamma') = \delta$};
				\draw		(1,0)	arc	(0:360:	1	and	2);
				\draw		(20,4)	arc	(90:-90:	2	and	4);
				\draw[dashed]	(20,4)	arc	(90:270:	2	and	4);
				\draw		(0,2)	..controls	(5,2)	and	(18,3)..	(19.8,4);
				\draw		(0,-2)	..controls	(5,-2)	and	(18,-3)..	(19.8,-4);
				\draw		(1,0)	..controls	(5,0)	and	(18,0)..	(22,0);
			\end{tikzpicture}
			\caption{\small{An example of regular annulus around $\gamma$ on $X$, where $\gamma$ is
a boundary component of $X$.}}
		\end{figure}
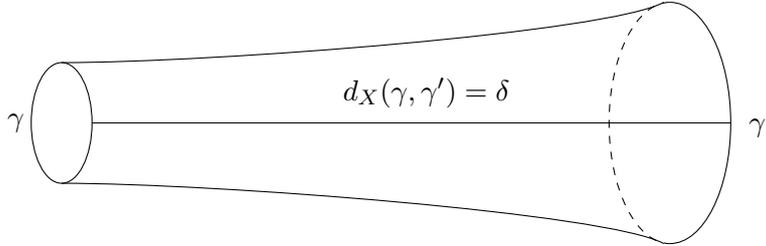

	\section{Proof of Theorem \ref{lemma:key}: the case where $\epsilon_0$ is sufficiently small}\label{ProofTheorem1}
	
		The proof of Theorem \ref{lemma:key} is separated  into two steps.
		Recall that $\epsilon_0$ is an upper bound for the lengths of all boundary components of $S$.
		In this section we will prove  Theorem \ref{lemma:key} in the case where $\epsilon_0$
is sufficiently small.
		We will consider the general case in next section.
	
		We assume that $\epsilon_0< e^{-1} \ln(1+\sqrt{2})$.
Here the constant $e^{-1} \ln(1+\sqrt{2})$ is chosen such that the width of some regular annulus neighborhood
around a boundary component of $S$ has an explicit lower bound.
Let $\epsilon_0' = \ln (1+\sqrt{2})$. Note that
$\epsilon_0< \epsilon_0' <2$.
Let $X_1$ and $X_2$ be two hyperbolic metrics in the $\epsilon_0$-relative part of $\mathcal{T}(S)$.
		We fix an essential arc $\beta \in \mathcal{B}(S)$. Denote the boundary curves where the two endpoints of $\beta$ lie by $\gamma$ and $\gamma'$.

  \subsection{The case where $\gamma \neq \gamma'$.}$\\$
% \bigskip

		In this subsection, we consider the case where $\gamma$ and $\gamma'$ are not the same boundary component of $S$.

There is a unique (isotopy class of) simple closed curve $\alpha$ which is homotopic to
the boundary of a regular neighborhood of $\beta \cup \gamma \cup \gamma'$.
		To simplify notation, for a given hyperbolic metric on $S$,
we will denote by $\alpha$, $\beta$, $\gamma$ and $\gamma'$ the geodesic representations of
the isotopy classes of $\alpha$, $\beta$, $\gamma$ and $\gamma'$, if  no confusion arises.

\begin{lemma}\label{lemma:separated}
        For $i=1,2$, we can take a regular annulus $A_i$ around $\gamma$ and a regular annulus $A_i'$ around $\gamma'$ on $X_i$ satisfying
         the following conditions:
        \begin{enumerate}
        \item  Denote the inner boundary of  $A_i$ by $C_i$ and the inner boundary of $A_i'$ by $C_i'$, then
         $\ell_{X_i}(C_i) =\ell_{X_i}(C_i')=\epsilon_0'$.
         \item The inner boundaries $C_i$ and $C_i'$ are disjoint.
            \end{enumerate}
\end{lemma}
\begin{proof}		
Denote by $X=X_1$. As we showed in Section 2, by the Collar Lemma,
there exist two disjoint regular annuli $A_{\eta(\ell_X(\gamma))}(\gamma)$ and $A_{\eta(\ell_X(\gamma'))}(\gamma')$.
        Let $\Delta_1$ and $\Delta_1'$ be the inner boundaries of  $A_{\eta(\ell_X(\gamma))}(\gamma)$ and $A_{\eta(\ell_X(\gamma'))}(\gamma')$, respectively.
By \eqref{eq:RegularAnnulis}, the length of $\Delta_1$ satisfies
        \begin{equation*}
            \begin{split}
                \ell_X(\Delta_1)
                    &= \ell_X(\gamma) \cosh(\eta(\ell_X(\gamma))) \\
                    &= \ell_X(\gamma) \cfrac{e^{\ell_X(\gamma)}+1}{e^{\ell_X(\gamma)}-1}.
            \end{split}
        \end{equation*}
   For $x>0$, we consider the function $$f_1(x) = x(e^x+1)/(e^x-1).$$
       It's easy to see that $f_1'(x)>0$ for all $x>0$ and $\lim_{x \to 0} f_1(x) = 2$.
        It follows that $f_1(x)>2$ for all $x>0$. In particular, we have $$\ell_X(\Delta_1) > 2 > \epsilon_0'.$$
        As a result, we can choose a regular annulus $A_1\subset A_{\eta(\ell_X(\gamma))}(\gamma)$ around $\gamma$ with inner boundary
        $C_1$ such that $\ell_X(C_1)=\epsilon_0'$.
       By the same argument, we can choose $C_1'$ to be the inner boundary of a regular annulus that is contained in $A_{\eta(\ell_X(\gamma'))}(\gamma')$.
      Since $A_{\eta(\ell_X(\gamma))}(\gamma)$ and $A_{\eta(\ell_X(\gamma'))}(\gamma')$ are disjoint, $C_1$ and $C_1'$ are disjoint.

By the same argument, we can choose
 $C_2$ and $C_2'$ on $X_2$ that are contained in disjoint regular annuli.

\end{proof}

        It follows from Lemma \ref{lemma:separated} that $C_1$ and $C_1'$ separate $\beta$ into three parts.
		Let ${\beta}_{1}^{A} = \beta \bigcap {A_1}$ and ${\beta'}_{1}^{A} = \beta \bigcap {A_1'}$ be the two terminal parts of $\beta$ and $\beta^{Q}_{1} = \beta \setminus \{ \beta_{1} ^{A}\bigcup {\beta'}_{1}^{A} \} $ be the middle part of $\beta$.
        We use similar notations $C_2$, $C_2'$, ${\beta}_{2}^{A}$, ${\beta'}_{2}^{A}$ and $\beta^{Q}_{2}$ for the hyperbolic structure $X_2$.
        Figure 4   shows the above notations.

		\begin{figure}[ht]\label{fig}
				\begin{tikzpicture}[scale=0.6]
					\draw		(-9.6,0)	node	{$\gamma$};
					\draw		(-9,0.5)	arc	(90:-90:	0.25	and 0.5);
					\draw		(-9,0.5)	arc	(90:270:	0.25	and 0.5);
					\draw		(-4.5,2)	node	{$C_i$};
					\draw		(-4.5,1.5)	arc	(90:-90:	0.5	and 1);
					\draw[dashed]	(-4.5,1.5)	arc	(90:270:	0.5	and 1);
					\draw		(-9,0.5)	..controls	(-8,0.5)	and	(-6,1)..	(-4.5,1.5);
	
					\draw		(9.6,0)	node	{$\gamma'$};
					\draw		(9,0.5)	arc	(90:-90:	0.25	and 0.5);
					\draw	[dashed](9,0.5)	arc	(90:270:	0.25	and 0.5);
					\draw		(4.5,2)	node	{$C_i'$};
					\draw		(4.5,1.5)	arc	(90:-90:	0.5	and 1);
					\draw[dashed]	(4.5,1.5)	arc	(90:270:	0.5	and 1);
					\draw		(9,0.5)	..controls	(8,0.5)	and	(6,1)..	(4.5,1.5);
					
					\draw		(-9,-0.5)	..controls	(-4.5,-0.5)	and	(4.5,-0.5)..	(9,-0.5);
	
					\draw		(0,4)	node	{$\alpha$};
					\draw		(-3,4)	arc	(180:360:	3	and	1);
					\draw[dashed](-3,4)	arc	(180:0:	3	and	1);
					
					\draw		(-4.5,1.5)	..controls	(-4,1.8)	and	(-3,2)..	(-3,4);
					\draw		(-3.5,5)	..controls	(-3.5,5)	and	(-3,5)..	(-3,4);
					\draw		(4.5,1.5)	..controls	(4,1.8)		and	(3,2)..	(3,4);
					\draw		(3.5,5)	..controls	(3.5,5)	and	(3,5)..	(3,4);
					
					\draw		(0,-1.2)	node	{$\beta$};
					\draw		(0,0.7)	node	{$\beta^{Q}_{i}$};
					\draw		(-6,0.6)	node	{$\beta^{A}_{i}$};
					\draw		(6,0.6)	node	{${\beta'}^{A}_{i}$};
					\draw		(-8.75,0)..controls	(-4.5,0.4)	and	(4.5,0.4)..	(9.25,0);
				\end{tikzpicture}
			\caption{\small{An illustration of a pair of pants on $X_i$ where $\gamma \neq \gamma'$ and $\epsilon_0 < e^{-1} \ln(1+\sqrt{2})$, for $i =1, 2$.  }}
		\end{figure}
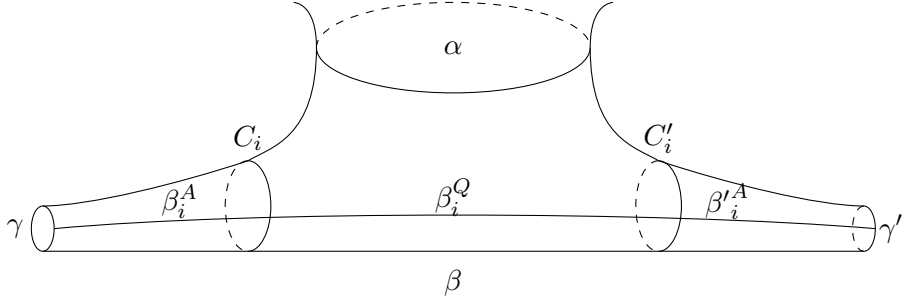

		The key point of our argument is to prove  that there exists a positive constant $K_1$ depending on $\epsilon_0$ such that
		\begin{equation}\label{eq:Diff_Boundary}
			\begin{split}
				\cfrac{\ell_{X_1}(\beta)}{\ell_{X_2}(\beta)}
				&\leq K_1 \cdot \max \{ 1, \
				\cfrac{\ell_{X_1}(\alpha)}{\ell_{X_2}(\alpha)}, \
				\cfrac{\ell_{X_2}(\gamma)}{\ell_{X_1}(\gamma)}, \
				\cfrac{\ell_{X_2}(\gamma')}{\ell_{X_1}(\gamma')}\}.
			\end{split}
		\end{equation}

\begin{figure}[ht]\label{fig:Case_1_Hexagon}
			\begin{tikzpicture}[scale=0.6]
				\draw		(-9.6,0)	node	{$c_i$};
				\draw		(-9.3,0.5)	..controls	(-9.2,0.3)	and	(-9.1,-0.1)..	(-9,-0.5);
				\draw		(-4.2,1)	node	{$c_0'$};
				\draw		(-5.2,2)	..controls	(-4.6,1.3)	and	(-4.5,0)..	(-4.5,-0.1);
	
				\draw		(9.6,0)	node	{$c_i'$};
				\draw		(9.3,0.5)	..controls	(9.2,0.3)	and	(9.1,-0.1)..	(9,-0.5);
				\draw		(4.2,1)	node	{$c_0'$};
				\draw		(5.2,2)	..controls	(4.6,1.3)	and	(4.5,0)..	(4.5,-0.1);
				
				\draw		(-7,-0.7)	node	{$d_i$};
				\draw		(7,-0.7)	node	{$d_i'$};
				\draw		(0,-0.5)	node	{$b_i'$};
				\draw		(-9,-0.5)	..controls	(-4.5,0.2)	and	(4.5,0.2)..	(9,-0.5);
	
				\draw		(0,5.3)	node	{$a_i$};
				\draw		(-3,5)	..controls	(-2,4.5)	and	(2,4.5)..	(3,5);
				
				\draw		(-9.3,0.5)	..controls	(-5,1.5)	and	(-3,3)..	(-3,5);
				\draw		(9.3,0.5)	..controls	(5,1.5)	and	(3,3)..	(3,5);
			\end{tikzpicture}
			\caption{\small{An example of the hexagon on $X_i$ when $\gamma \neq \gamma'$ and $\epsilon_0 < e^{-1} \ln(1+\sqrt{2})$, for $i = 1, 2$.  }}
		\end{figure}
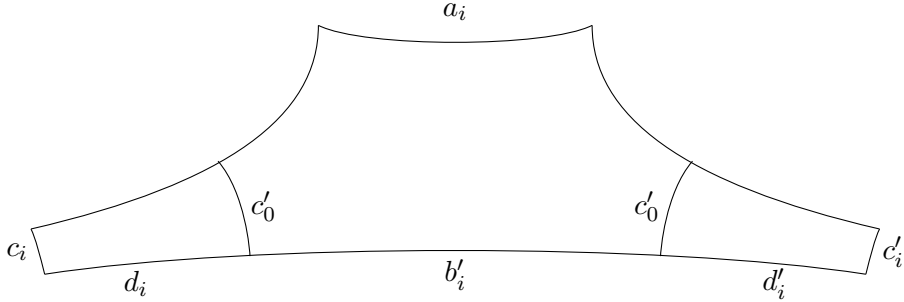

       Let us explain more explicitly.
        As it is shown in Figure 5, by cutting the pair of pants along three geodesic arcs, each of
         which is perpendicular to a pair of boundary components, we have two right-angled hexagons.
        By symmetry, we only need to consider one of them.
		
For  the sake of simplicity, we denote $\ell_{X_i}$ by $\ell_i$, for $i = 1,2$.
		Let $a_i = \ell_i(\alpha)/2$, $b_i = \ell_i(\beta)$, $c_i = \ell_i(\gamma)/2$, $c'_i = \ell_i(\gamma')/2$, $d_i = \ell_i(\beta^{A}_{i})$, $d'_i = \ell_i({\beta'}^{A}_{i})$, $b'_i = \ell_i(\beta^{Q}_{i})$, for $i = 1, 2$.
        And let $c_0' = \epsilon_0'/2$.
		Then $b_i = d_i + b'_i + d'_i$, for $i = 1, 2$.

        With the above notations, we have
		\begin{equation*}\label{eq:C1_main}
			\begin{split}
				\cfrac{\ell_1(\beta)}{\ell_2(\beta)} = \cfrac{b_1}{b_2}
				= \cfrac{b_1'+d_1+d_1'}{b_2'+d_2+d_2'}
%				&\leq \max\{\cfrac{b_1'}{b_2'+d_2+d_2'},\ \cfrac{d_1}{b_2'+d_2+d_2'},\ \cfrac{d_1'}{b_2'+d_2+d_2'} \} \\
				\leq \max\{\cfrac{b_1'}{b_2'},\ \cfrac{d_1}{d_2},\ \cfrac{d_1'}{d_2'} \}.
			\end{split}
		\end{equation*}

		To prove \eqref{eq:Diff_Boundary}, it suffices to control $b_1'/b_2'$, $d_1/d_2$ and $d_1'/d_2'$ by the ratios of the lengths of
$\alpha, \gamma$ and $\gamma'$.
		This will be done in Lemma \ref{lemma:C1_Pb_Final} and Lemma \ref{lemma:C1_Pd_Final} below.
As soon as Lemma \ref{lemma:C1_Pb_Final} and Lemma \ref{lemma:C1_Pd_Final} are proved,
\eqref{eq:Diff_Boundary} is a direct corollary, see Proposition \ref{lemma:Diff_Boundary}.
		
\begin{example}[An exceptional case]\label{example1} If  $S=S_{0,1,2}$, that is,
the surface is homeomorphic to a pair of pants with one puncture and two boundary components,
then $a_i = \ell_i(\alpha)/2=0$. In this case, the ratio $\frac{a_1}{a_2}$ in the following discussions would not make sense.

To avoid this difficulty, we can take two sequences of pairs of pants $\left( X_{1,n}\right)_{n=1}^\infty, \left( X_{2,n}\right)_{n=1}^\infty$ (we denote their boundary components by $\alpha, \gamma, \gamma'$ as above) such that
$$ \ell_{X_{i,n}}(\gamma)=\ell_{X_i}(\gamma), \ell_{X_{i,n}}(\gamma')=\ell_{X_i}(\gamma'), \ell_{X_{i,n}}(\alpha)=\frac{1}{n}, \ i=1,2.$$
Since the constants in the following lemmas are independent of $n$, by taking a limit as $n$ goes to infinity,
we will get the same results (all the following lemmas are true in such a special case if we set $\frac{0}{0}=1$).

Note that the same argument applies to $S=S_{0,2,1}$, that is,
the surface is homeomorphic to a pair of pants with two punctures and one boundary component,
which we will consider in Section \ref{3.3}.
\end{example}

By Example \ref{example1}, we can assume that $a_i>0, \ i=1,2$. We first consider the ratio $b_1'/b_2'$.
		\begin{lemma}\label{lemma:C1_Pb_Final}
			There exists a positive constant $K_1'$ depending on $\epsilon_0$ such that
			\begin{equation}\label{eq:C1_Pb_Final}
					\cfrac{b_1'}{b_2'}
					\leq K_1' \cdot \max \{ 1,\ \cfrac{a_1}{a_2} \}.
			\end{equation}
		\end{lemma}

		\begin{proof}

           The proof of this lemma will use Lemma \ref{lemma:bi} and Lemma \ref{lemma:difference} below.

           \begin{lemma}\label{lemma:bi}
         There is a uniform positive lower bound for $b_i', \ i=1,2$.
           \end{lemma}
\begin{proof}
            Recall that the regular annulus $A_{\eta(\ell_1(\gamma))}(\gamma)$ contains $C_1$ and the length of the inner boundary of $A_{\eta(\ell_1(\gamma))}(\gamma)$  is greater than $2$.
            We can take another regular annulus around $\gamma$ which is isometrically embedded in
             $A_{\eta(\ell_1(\gamma))}(\gamma)$ and which has a inner boundary, denoted by $\widetilde{C}_1$, with length equal to $2$.
            Denote by $e_1$ the distance between $\gamma$ and $\widetilde{C}_1$.
            It can be seen from Figure 6   that $b_1' \geq (e_1 - d_1) + (e_1' - d_1')$.

			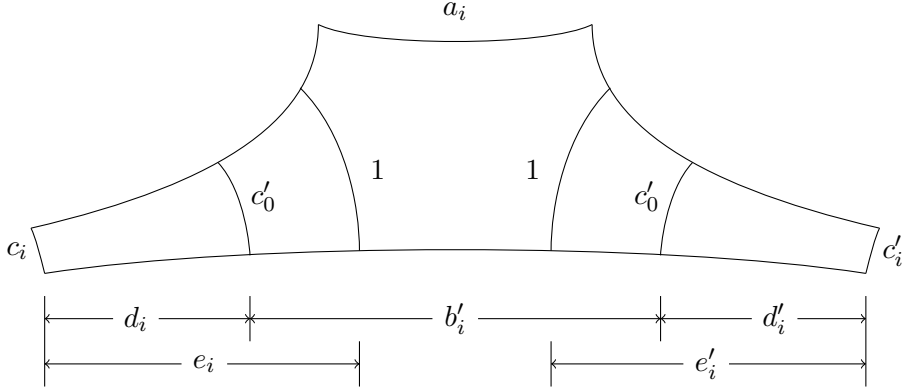
\begin{figure}[ht]\label{fig:C1_Pb_1st}
				\begin{tikzpicture}[scale=0.6]
					\draw		(-9.6,0)	node	{$c_i$};
					\draw		(-9.3,0.5)	..controls	(-9.2,0.3)	and	(-9.1,-0.1)..	(-9,-0.5);
					\draw		(-4.2,1.2)	node	{$c_0'$};
					\draw		(-5.2,1.95)	..controls	(-4.6,1.3)	and	(-4.5,0)..	(-4.5,-0.1);
					\draw		(-1.7,1.8)	node	{$1$};
					\draw		(-3.4,3.6)	..controls	(-2.4,2.6)	and	(-2.1,1)..	(-2.1,0);
	
					\draw		(9.6,0)	node	{$c_i'$};
					\draw		(9.3,0.5)	..controls	(9.2,0.3)	and	(9.1,-0.1)..	(9,-0.5);
					\draw		(4.2,1.2)	node	{$c_0'$};
					\draw		(5.2,1.95)	..controls	(4.6,1.3)	and	(4.5,0)..	(4.5,-0.1);
					\draw		(1.7,1.8)	node	{$1$};
					\draw		(3.4,3.6)	..controls	(2.4,2.6)	and	(2.1,1)..	(2.1,0);

					\draw		(-9,-1) 	-- 	(-9,-3);
					\draw		(-4.5,-1)	-- 	(-4.5,-2);
					\draw		(-2.1,-2) 	-- 	(-2.1,-3);
					\draw		(9,-1) 	-- 	(9,-3);
					\draw		(4.5,-1)	-- 	(4.5,-2);
					\draw		(2.1,-2)	-- 	(2.1,-3);
											
					\draw	[<-]	(-9,-1.5)	-- 	(-7.5,-1.5);
					\draw	[->]	(-6.5,-1.5)	-- 	(-4.5,-1.5);
					\draw		(-7,-1.5)	node	{$d_i$};
					\draw	[<-]	(9,-1.5)	-- 	(7.5,-1.5);
					\draw	[->]	(6.5,-1.5)	-- 	(4.5,-1.5);
					\draw		(7,-1.5)	node	{$d_i'$};
					\draw	[<-]	(-4.5,-1.5)	-- 	(-0.5,-1.5);
					\draw	[->]	(0.5,-1.5)	-- 	(4.5,-1.5);
					\draw		(0,-1.5)	node	{$b_i'$};
					\draw	[<-]	(-9,-2.5)	-- 	(-6,-2.5);
					\draw	[->]	(-5,-2.5)	-- 	(-2.1,-2.5);
					\draw		(-5.5,-2.5)	node	{$e_i$};
					\draw	[<-]	(9,-2.5)	-- 	(6,-2.5);
					\draw	[->]	(5,-2.5)	-- 	(2.1,-2.5);
					\draw		(5.5,-2.5)	node	{$e_i'$};

					\draw		(-9,-0.5)	..controls	(-4.5,0.2)	and	(4.5,0.2)..	(9,-0.5);
	
					\draw		(0,5.3)	node	{$a_i$};
					\draw		(-3,5)	..controls	(-2,4.5)	and	(2,4.5)..	(3,5);
					
					\draw		(-9.3,0.5)	..controls	(-5,1.5)	and	(-3,3)..	(-3,5);
					\draw		(9.3,0.5)	..controls	(5,1.5)	and	(3,3)..	(3,5);
				\end{tikzpicture}
				\caption{\small{An example for $e_i$ and $e_i'$, for $i = 1, 2$. }}
			\end{figure}
			
            It suffices to give a lower bound for $e_1-d_1$. By \eqref{eq:RegularAnnulis} and \eqref{eq:Arcosh}, we have \begin{equation*}
                e_1 - d_1 = \ln \cfrac{1/c_1 + \sqrt{(1/c_1)^2 -1}}{c_0'/c_1 + \sqrt{(c_0'/c_1)^2 -1}},
            \end{equation*}
            where  $c_1< \frac{\ln ( 1 + \sqrt{2})} {2e}$.  Consider the function
            \begin{equation*}
                f_2(y) = ( y + \sqrt{y^2-1} ) / ( c_0' y + \sqrt{{c_0'}^2 y^2 -1} ), \  y >2 e / \ln ( 1 + \sqrt{2}).
            \end{equation*}
            By the fact  $f_2'(y) <0$ and $y > 2 e / \ln ( 1 + \sqrt{2})$, we have
            $$e_1 - d_1 = f_2(1/c_1) \geq f_2(2 e / \ln(1+\sqrt{2})) = 4 / \ln(1+\sqrt{2}) >0.$$
            By the same argument we have $e_1' - d_1' \geq 4 / \ln(1+\sqrt{2})$.
            Let $M_0= 8 / \ln(1+\sqrt{2})$, then  we have (the same estimation for $b_2'$)
			\begin{equation}\label{eq:C1_Pb_1st}
				b_i' \geq M_0, \text{ for } i=1,2.
			\end{equation}		
		\end{proof}

           Next we will give a upper bound for the difference between $a_i$ and $b_i'$, $i=1,2$.

           \begin{lemma} \label{lemma:difference}There is a constant $D_1$ depending on $\epsilon_0$ such that
           \begin{equation}\label{eq:C1_Pb_5nd}
           |a_i-b_i'|\leq D_1, \ i=1,2.
           \end{equation}
           \end{lemma}
			\begin{proof} The method used here is similar to that of \cite{Rafi}.

			Since $c_1 \leq \epsilon_0/2$ and $c_1' \leq \epsilon_0/2$,
there exists a constant $k_1$ depending on $\epsilon_0$ such that
$$c_1 < \sinh c_1 < k_1 c_1 \  \mathrm{and}  \  c_1' < \sinh c_1' < k_1 c_1'.$$
			By \eqref{eq:RegularAnnulis}, we have
$$
c_1 \cosh d_1 = c_1' \cosh d_1' = c_0' = \ln(1+\sqrt{2})/2.
$$
			Then we have
			\begin{equation*}
				\begin{split}
					&\sinh c_1 \cdot \sinh c_1' \cdot \cosh (b_1'+d_1+d_1') \\
					&> c_1 \cdot c_1' \cdot \cfrac{e^{b_1'+d_1+d_1'}}{2} \
					= \cfrac{e^{b_1'}}{2} \cdot c_1 e^{d_1} \cdot c_1' e^{d_1'} \\
					&> \cfrac{e^{b_1'}}{2} \cdot c_1 \ \cosh{d_1} \cdot c_1' \ \cosh{d_1'} \
					= \cfrac{1}{2} \ c_0'^2 \ e^{b_1'},
				\end{split}
			\end{equation*}
			and
			\begin{equation*}
				\begin{split}
					&\sinh c_1 \cdot \sinh c_1' \cdot \cosh (b_1'+d_1+d_1') \\
					&< k_1 \ c_1 \cdot k_1 \ c_1' \cdot e^{b_1'+d_1+d_1'} \
					= 4 \ k_1^2 \ e^{b_1'} \cdot c_1 \ \cfrac{e^{d_1}}{2} \cdot c_1' \  \cfrac{e^{d_1'}}{2} \\
					&< 4 \ k_1^2 \ e^{b_1'} \cdot c_1 \ \cosh{d_1} \cdot c_1' \ \cosh{d_1'} \
					= 4 {k_1}^2 {c_0'}^2 \ e^{b_1'}.
				\end{split}
			\end{equation*}

			Let $M_1 = \max \{  2/{c_0'}^2, 4 k_1^2 {c_0'}^2\}$.
          It follows that
            $$
            M_1^{-1} e^{b_1'}\leq \sinh c_1 \cdot \sinh c_1' \cdot \cosh (b_1'+d_1+d_1')\leq M_1 e^{b_1'}.$$
	Combining the above inequality with \eqref{eq:RightAngledHexagon}, we have
			\begin{equation*}
				\begin{split}
					e^{a_1}
					&\leq 2 \cosh{a_1} \
					< 2 ( \cosh{a_1} + \cosh{c_1} \cdot \cosh{c_1'} ) \\
					&= 2 \cdot \sinh c_1 \cdot \sinh c_1' \cdot \cosh ( b_1' + d_1 + d_1' ) \\
					&\leq 2 M_1 \cdot e^{b_1'}.
				\end{split}
			\end{equation*}
			On the other hand, we have
			\begin{equation*}
				\begin{split}
					\cosh c_1 \cdot \cosh c_1'
					&< \cosh c_1 \cosh c_1' + \sinh c_1 \sinh c_1' \\
					&= \cosh (c_1 + c_1') \
					< \cosh (\cfrac{\epsilon_0}{2} + \cfrac{\epsilon_0}{2}) \\
					&< \cosh \epsilon_0 \cdot \cosh a_1.
				\end{split}
			\end{equation*}
Applying \eqref{eq:RightAngledHexagon} again, we have
			\begin{equation*}
				\begin{split}
					e^{a_1}
					&\geq \cosh{a_1} \
					= (1+\cosh\epsilon_0)^{-1} (\cosh{a_1} + \cosh\epsilon_0 \cosh{a_1}) \\
					&> (1+\cosh\epsilon_0)^{-1} (\cosh{a_1} + \cosh c_1 \cdot \cosh c_1')  \\ %\textcolor[rgb]{1.00,0.00,0.00}{(Why \ this \  inequality \  is \  true?)}\\
					&= (1+\cosh\epsilon_0)^{-1} \cdot \sinh c_1 \cdot \sinh c_1' \cdot \cosh ( b_1' + d_1 + d_1' ) \\
					&\geq (1+\cosh\epsilon_0)^{-1} M_1^{-1} \cdot e^{b_1'}.
				\end{split}
			\end{equation*}
In conclusion, we have
            $$
            (1+\cosh\epsilon_0)^{-1} M_1^{-1} \cdot e^{b_1'}\leq e^{a_1}\leq 2 M_1 \cdot e^{b_1'}
            $$
or, equivalently,
            $$
            (1+\cosh\epsilon_0)^{-1} M_1^{-1} \leq e^{a_1-b_1'}\leq 2 M_1.
            $$
Setting $D_1= \max \{ |\ln (2 M_1)|,\ |\ln (M_1\cdot (1+\cosh\epsilon_0))|\}$, then we have
            \begin{equation*}\label{eq:C1_Pb_2nd}
				|a_1 - b_1'| \leq D_1.
			\end{equation*}
By the same proof we also have
\begin{equation*}\label{eq:C1_Pb_2nd}
				|a_2 - b_2'| \leq D_1.
			\end{equation*}
\end{proof}

			We continue with our proof of Lemma \ref{lemma:C1_Pb_Final}. Let $M > 2 D_1$ be a sufficiently large positive number. The remaining discussion is separated into several different cases.

\bigskip
	
			\begin{case}{1: $b_i'\geq M, \ i=1,2$.}

\bigskip

In this case, using  \eqref{eq:C1_Pb_5nd}, we have (for $i=1,2$)
				\begin{equation*}
						\cfrac{a_i}{b_i'} \leq \cfrac{b_i' + D_1}{b_i'} < 1 + \cfrac{D_1}{M}< \frac{3}{2}
				\end{equation*}
				and
				\begin{equation*}
						\cfrac{a_i}{b_i'} \geq \cfrac{b_i' - D_1}{b_i'} > 1 - \cfrac{D_1}{M} > \cfrac{1}{2}.
				\end{equation*}
That is
$$
 \frac 1 2 \leq \cfrac{a_i}{b_i'} \leq \frac 3 2.
$$
	It follows that
					\begin{equation*}
						\cfrac{b_1'}{b_2'}
						\leq 3 \cdot \cfrac{a_1}{a_2}.
					\end{equation*}	
\end{case}

\bigskip
	
			\begin{case}{2: $b_i' \leq M$ and $a_i > \epsilon_0, i=1,2$.}

\bigskip

									Combing with \eqref{eq:C1_Pb_1st} and \eqref{eq:C1_Pb_5nd}, we have $M_0 \leq b_i' \leq M$ and $\epsilon_0 \leq a_i \leq b_i' + D_1 \leq M + D_1$.
				It follows that
				\begin{equation*}
						\cfrac{2 M_0}{3 M} < \cfrac{M_0}{M+D_1} \leq \cfrac{b_i'}{a_i} \leq \cfrac{M}{\epsilon_0}.
				\end{equation*}
		In this case
\begin{equation*}
						\cfrac{b_1'}{b_2'}
						\leq \cfrac{M}{\epsilon_0} \cdot \cfrac{3M}{2 M_0} \cdot \cfrac{a_1}{a_2}.
					\end{equation*}	
			
	\end{case}

\bigskip
	
			\begin{case}{3: $b_1'>M$ and $b_2'\leq M, a_2>\epsilon_0$.}

\bigskip	
				
	 It follows from the estimations in {Case 1} and {Case 2}
that
                     \begin{equation*}
						\cfrac{b_1'}{b_2'}
						\leq  \cfrac{3M}{M_0} \cdot \cfrac{a_1}{a_2}.
					\end{equation*}	
\end{case}

\bigskip
	
			\begin{case}{4: $b_2'>M$ and $b_1'\leq M, a_1>\epsilon_0$.}

\bigskip

In this case, we have the same conclusion as in {Case 3}.
					
\end{case}

\bigskip

			\begin{case}{5: $b_1'>M$ and $b_2'\leq M, a_2\leq \epsilon_0$.}

\bigskip
					By \eqref{eq:C1_Pb_1st}, we have
					\begin{equation*}
						\begin{split}
							\cfrac{b_1'}{b_2'}
							&\leq \cfrac{b_1'}{M_0}
							\leq \cfrac{2\cdot a_1}{M_0}
							= \cfrac{2\cdot a_2}{M_0} \cdot \cfrac{a_1}{a_2}
							\leq \cfrac{2\cdot \epsilon_0}{M_0} \cdot \cfrac{a_1}{a_2}.
						\end{split}
					\end{equation*}
		
\end{case}

\bigskip
	
			\begin{case}{6:  $b_1'\leq M, a_1> \epsilon_0$ and $b_2'\leq M, a_2\leq \epsilon_0$.}

\bigskip

					By \eqref{eq:C1_Pb_1st}, we have
					\begin{equation*}
						\begin{split}
							\cfrac{b_1'}{b_2'}
							&\leq \cfrac{b_1'}{M_0}
							\leq \cfrac{M}{\epsilon_0}\cfrac{\cdot a_1}{M_0}
							=\cfrac{M}{\epsilon_0} \cfrac{\cdot a_2}{M_0} \cdot \cfrac{a_1}{a_2}
							\leq \cfrac{M}{\epsilon_0}\cfrac{\cdot \epsilon_0}{M_0} \cdot \cfrac{a_1}{a_2}= \cfrac{a_1}{a_2}.
						\end{split}
					\end{equation*}
	\end{case}

\bigskip
	
			\begin{case}{7:  $b_1'\leq M, a_1\leq \epsilon_0$ and $b_2'\leq M, a_2\leq \epsilon_0$.}

\bigskip	
		
			By \eqref{eq:C1_Pb_1st}, we have
				\begin{equation*}
					\cfrac{b_1'}{b_2'}
					\leq \cfrac{M}{M_0}.
				\end{equation*}
		In this case, it is obvious that
			\begin{equation*}
					\cfrac{b_1'}{b_2'}
					\leq\cfrac{M}{M_0} \cdot \max \{ 1,\cfrac{a_1}{a_2} \}.
			\end{equation*}
\end{case}

\bigskip

The other two remaining cases, that is, $b_2'>M, b_1'\leq M, a_1\leq \epsilon_0$ and $b_2'\leq M, a_2> \epsilon_0, b_1'\leq M, a_1\leq \epsilon_0$, can be reduced to {Case 5} and {Case 6}.
By choosing $K_1'=\max \{3, \frac{3M^2}{2M_0\epsilon_0}, \frac{3M}{M_0}, \frac{2\epsilon_0}{M_0}\}$, we
complete the proof of Lemma \ref{lemma:C1_Pb_Final}.
\end{proof}

		Next we will consider the ratio $d_1/d_2$.

		\begin{lemma}\label{lemma:C1_Pd_Final} The ratio $d_1/d_2$ has an upper bound given by
			\begin{equation}\label{eq:C1_Pd_Final}
				\cfrac{d_1}{d_2}
				\leq 2 \cdot \max \{ 1, \cfrac{c_2}{c_1} \} \ .
			\end{equation}
		\end{lemma}

		\begin{proof}  As $c_0' = (\ln{(1+\sqrt{2})})/2$ and $\epsilon_0 < e^{-1} \ln{(1+\sqrt{2})}$, we have
$$c_0'/c_1 \geq \epsilon_0'/\epsilon_0 > e > 1.$$
 By \eqref{eq:RegularAnnulis}, we have $$c_1 \cosh d_1 =c_0'.$$
           By \eqref{eq:Arcosh}, we have $$d_1 = \text{arcosh} (c_0'/c_1) = \ln( c_0'/c_1 + \sqrt{(c_0'/c_1)^2 -1}).$$

            Note that for any $x > e$, $\ln (2 x) \leq 2 \ln x$.
			Since $\sqrt{(c_0'/c_1)^2 -1} \leq c_0'/c_1$ and $c_0'/c_1 > e$, we have
$$
d_1 = \ln( c_0'/c_1 + \sqrt{(c_0'/c_1)^2 -1}) \leq \ln (2 c_0'/c_1) \leq 2 \cdot \ln (c_0'/c_1).$$
Since $d_1 = \ln( c_0'/c_1 + \sqrt{(c_0'/c_1)^2 -1}) \geq \ln (c_0'/c_1)$, we have
						\begin{equation*}\label{eq:C1_Pd_2nd}
				\ln c_0' - \ln c_1 \leq d_1 \leq 2 \cdot (\ln c_0' - \ln c_1) .
			\end{equation*}

The same discussion implies

\begin{equation*}\label{eq:C1_Pd_2*nd}
				\ln c_0' - \ln c_2 \leq d_2 \leq 2 \cdot (\ln c_0' - \ln c_2) \ .
			\end{equation*}

			As a result, we have
			\begin{equation*}
				\cfrac{d_1}{d_2}
				\leq  2 \cdot \cfrac{ \ln c_0' - \ln c_1 }{\ln c_0' - \ln c_2 }				.
			\end{equation*}

			If $c_2 \leq c_1$,  we have
$$\frac{\ln c_0' - \ln c_1}{\ln c_0' - \ln c_2} \leq 1.$$

Now suppose that $c_2 > c_1$.
				Let $f_3(x) = x^{-1} \ln x$. Then $f_3'(x) = (1 - \ln x)/{x^2}$. We know that $f_3'(x) \leq 0$ as $x \geq e$.
				Since $\frac {c_0'} {c_i} \geq e, i=1,2$, and $\frac {c_0'} {c_1} > \frac {c_0'} {c_2}.$ It follows that $f_3(\frac {c_0'} {c_1}) < f_3(\frac {c_0'} {c_2})$. This implies
				\begin{equation*}
					\cfrac{\ln c_0' - \ln c_1}{\ln c_0' - \ln c_2}
					\leq \cfrac{c_2}{c_1}.
				\end{equation*}
		
	The above discussions lead to the following inequality:
			\begin{equation*}
				\cfrac{d_1}{d_2} \leq 2 \cdot \max \{ 1, \cfrac{c_2}{c_1} \}.
			\end{equation*}
		
		\end{proof}

By the same discussion as above, we have
\begin{lemma}\label{lemma:C1_Pd*_Final} The ratio $d_1'/d_2'$ has an upper bound given by
						\begin{equation}\label{eq:C1_Pd_Final}
				\cfrac{d_1'}{d_2'}
				\leq 2 \cdot \max \{ 1, \cfrac{c_2'}{c_1'} \}.
			\end{equation}
		\end{lemma}

		\begin{proposition}\label{lemma:Diff_Boundary}
			Let $\epsilon_0 < e^{-1} \ln(1+\sqrt{2})$.  For any
essential arc $\beta \in \mathcal{B}(S)$ whose endpoints lie on different boundary components $\gamma$ and $\gamma'$ of $S$, let $\alpha$ be the associated simple closed curve homotopic to the boundary of a regular neighborhood of $\beta \cup \gamma \cup \gamma'$. Then there exists a positive number $K_1$ depending on $\epsilon_0$ such that  inequality \eqref{eq:Diff_Boundary} holds for any $X_1, X_2$ in the $\epsilon_0$-relative part of $\mathcal{T}(S)$.
		\end{proposition}
		
\begin{proof} We apply the results of Lemma \ref{lemma:C1_Pb_Final}
and Lemma \ref{lemma:C1_Pd_Final} and the notations in their proof.
It follows that the ratio of $\ell_1(\beta)$ and $\ell_2(\beta)$ satisfies
			\begin{equation*}
				\begin{split}
					\cfrac{\ell_1(\beta)}{\ell_2(\beta)}
					&\leq \max \{ \cfrac{b_1'}{b_2'},\ \cfrac{d_1}{d_2},\ \cfrac{d_1'}{d_2'} \} \\
					&\leq \max \{K_1' \cdot \max \{ 1,\cfrac{a_1}{a_2} \},\ 2 \cdot \max \{ 1,\cfrac{c_2}{c_1}\},\ 2 \cdot \max \{ 1,\cfrac{c_2'}{c_1'} \} \} \\
					&\leq K_1 \cdot \max \{ 1, \cfrac{a_1}{a_2}, \cfrac{c_2}{c_1}, \cfrac{c_2'}{c_1'} \} \\
					&=  K_1 \cdot \max \{ 1, \cfrac{\ell_1(\alpha)}{\ell_2(\alpha)}, \cfrac{\ell_2(\gamma)}{\ell_1(\gamma)}, \cfrac{\ell_2(\gamma')}{\ell_1(\gamma')} \},
				\end{split}
			\end{equation*}
			where $K_1 = \max\{K_1',\ 2\}$ only depend on $\epsilon_0$.
		\end{proof}

\subsection{The case where $\gamma=\gamma'$.}\label{3.3}
$\newline$

% \bigskip

Now we consider the case where $\gamma = \gamma'$.
		In this case, we denote by $\gamma$ the boundary component of $S$ where the two endpoints of $\beta$ lie. 		

        Consider a regular neighborhood of $\beta \cup \gamma$.
        It  is homotopic to a pair of pants
        whose boundary components consist of $\gamma$ and  two other simple closed curves, denoted by $\alpha$ and $\alpha'$.

        We will prove an analogue of inequality \eqref{eq:Diff_Boundary},
        that is, there exists a positive constant $K_2$ depending on $\epsilon_0$ such that
        \begin{equation}\label{eq:Same_Boundary}
            \cfrac{\ell_{X_1}(\beta)}{\ell_{X_2}(\beta)}
            \leq K_2 \cdot \max\{1,\ \cfrac{\ell_{X_1}(\alpha)}{\ell_{X_2}(\alpha)},\ \cfrac{\ell_{X_1}(\alpha')}{\ell_{X_2}(\alpha')},\ \cfrac{\ell_{X_2}(\gamma)}{\ell_{X_1}(\gamma)}\}.
        \end{equation}

\begin{remark}  \label{remark:alpha}
By Example \ref{example1}, we can assume that one of the curves $\gamma$ and $\gamma'$ is not homotopic to a puncture.
        Without loss of generality, we may suppose that $\ell_{X_1}(\alpha)\geq \ell_{X_1}(\alpha')$.
        Note that  $\alpha'$  maybe homotopic to a puncture. In this case, we shall identify a puncture
        with a simple closed geodesic with length zero and let $\frac{0}{0}=1$.
\end{remark}

       For $X_i,  \ i = 1, 2$, let $C_i$ be the inner boundary of the regular annulus around $\gamma$ with length $\ell_{X_i}(C_i) = \epsilon_0'$
       (the existence of such a regular annulus is given by Lemma \ref{lemma:separated}).
        Then $C_i$ separates $\beta$ into three parts $\beta^{A}_{i}$, $\beta^{Q}_{i}$ and ${\beta'}^{A}_{i}$, for $i = 1, 2$.
	See Figure 7.

		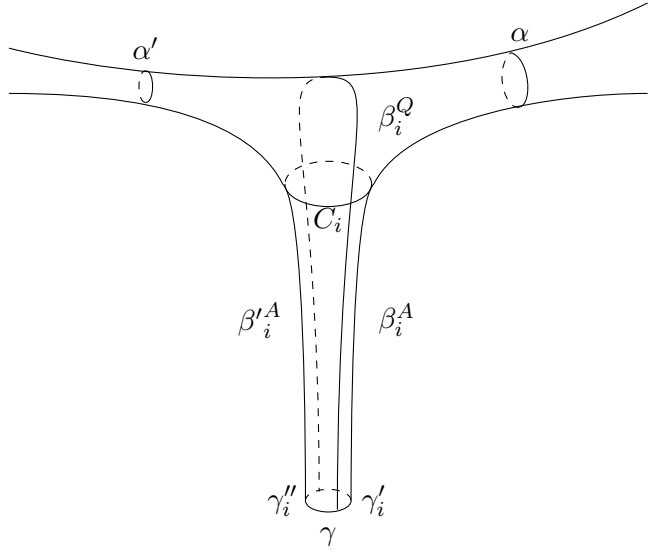
\begin{figure}[ht]\label{fig:Case_2_Surface}
			\begin{tikzpicture}[scale=0.6]
				\draw		(0,-0.8)	node	{$\gamma$};
				\draw		(1,0)	node	{$\gamma_i'$};
				\draw		(-1,0)	node	{$\gamma_i''$};
				\draw		(-0.5,0)	arc	(180:360:	0.5	and	0.25);
				\draw[dashed](0.5,0)	arc	(0:180:	0.5	and	0.25);				
				\draw		(0,6.2)	node	{$C_i$};
				\draw		(-0.95,7)	arc	(180:360:	0.95	and	0.5);
				\draw[dashed](0.95,7)	arc	(0:180:	0.95	and	0.5);
				
				\draw		(-0.5,0)	..controls	(-0.5,1)	and	(-0.5,6.5)..	(-1,7);
				\draw		(0.5,0)	..controls	(0.5,1)	and	(0.5,6.5)..	(1,7);
				
				\draw		(-1,7)	..controls	(-1.5,8)	and	(-3,9)..	(-7,9);
				\draw		(1,7)	..controls	(1.75,8.5)	and	(5,9)..	(7,9);
				\draw		(7,11)	..controls	(3,9)		and	(-3,9)..	(-7,10);
				
				\draw		(-4,10)	node	{$\alpha'$};
				\draw		(-4,8.8)	..controls	(-3.8,8.8)	and	(-3.8,9.5)..	(-4,9.5);
				\draw	[dashed](-4,8.8)	..controls	(-4.2,8.8)	and	(-4.2,9.5)..	(-4,9.5);
				\draw		(4.2,10.3)	node	{$\alpha$};
				\draw		(4.2,8.7)	..controls	(4.5,8.8)	and	(4.4,9.8)..	(4,9.9);
				\draw	[dashed](4.2,8.7)	..controls	(3.9,8.6)	and	(3.6,9.5)..	(4,9.9);
				
				\draw		(1.5,4)	node	{$\beta^{A}_{i}$};
				\draw		(-1.5,4)	node	{${\beta'}^{A}_{i}$};
				\draw		(1.5,8.5)	node	{$\beta^{Q}_{i}$};
				\draw		(0,9.36)		..controls	(1.3,9.36)		and	(0.2,8)..	(0.2,-0.2);
				\draw	[dashed](0,9.36)		..controls	(-1.3,9.36)		and	(-0.2,8)..	(-0.2,0.2);
			\end{tikzpicture}
			\caption{\small{An example of the pair of pants when $\gamma = \gamma'$ and $\epsilon_0 < e^{-1} \ln(1+\sqrt{2})$, for $i = 1, 2$. }}
		\end{figure}

      One can see from Figure 7   that the endpoints of $\beta$ separate the geodesic $\gamma$ into two parts, denoted by
       $\gamma_i'$ and $\gamma_i''$. Note that
        $\gamma_i' \cup \beta$ (resp. $\gamma_i'' \cup \beta$) is isotopic to $\alpha$ (resp. $\alpha'$), for $i = 1, 2$.

       By cutting each pair of pants along the three perpendicular geodesic arcs connecting the boundary components,
        we have two symmetric right-angled hexagons on $X_i$, for $i=1,2$.
        We consider one of them for $i=1,2$, as we shown in Figure 8.
		To simplify notation, we denote $\ell_{X_i}$ by $\ell_i$ and let $c_0' = \epsilon_0'/2$, $b_i = \ell_i(\beta^{Q}_{i})/2$,
$d_i = \ell_i(\beta^{A}_{i}) = \ell_i({\beta'}^{A}_{i})$, $a_{i} = \ell_i(\alpha)/2$,
$a_{i}' = \ell_i(\alpha')/2$, $c_i' = \ell_i(\gamma_i')/2$ and $c_i'' = \ell_i(\gamma_i'')/2$, for $i = 1, 2$.
See Figure 8. Since $l_i(\alpha) \geq l_i(\alpha')$, we have  $a_{i} \geq a_{i}'$, for $i =1,\ 2$.

		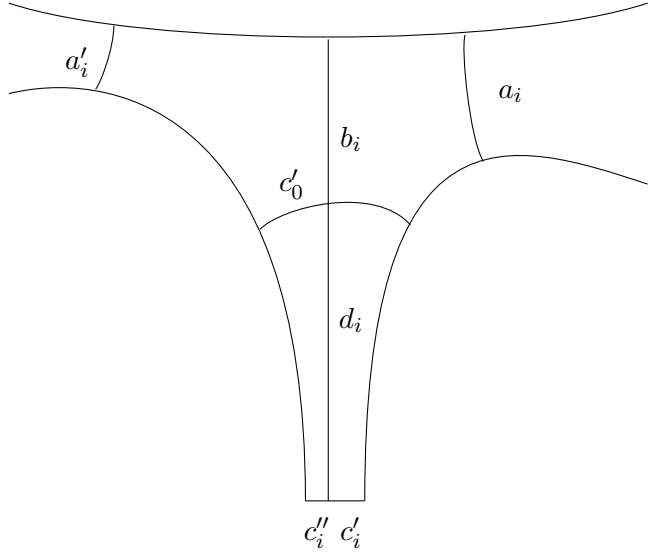
\begin{figure}[ht]\label{fig:Case_2_Pentagon}
			\begin{tikzpicture}[scale=0.6]
				\draw		(0.5,-0.7)	node	{$c_i'$};
				\draw		(-0.25,-0.7)	node	{$c_i''$};
				\draw		(-0.5,0)		..controls	(0.1,0)	and	(0.4,0)..	(0.8,0);
				\draw		(-0.8,7)	node	{$c_0'$};
				\draw		(-1.5,6)		..controls	(-1,6.5)	and	(1,7)..	(1.8,6.1);
				
				\draw		(0.5,4)	node	{$d_i$};
				\draw		(0.5,8)	node	{$b_i$};
				\draw		(0.8,0)	..controls	(0.8,9)	and	(4,8)..	(7,7);
				\draw		(-0.5,0)	..controls	(-0.5,9)	and	(-5,9.5)..	(-7,9);
				\draw		(-7,11)	..controls	(-4,10)	and	(4,10)..	(7,11);
				\draw		(0,0)		..controls	(0,1)	and	(0,8)..	(0,10.2);
				
				\draw		(4,9)	node	{$a_{i}$};
				\draw		(3.4,7.5)	..controls	(3.1,8)	and	(2.9,10)..	(3,10.3);
				\draw		(-5.5,9.7)	node	{$a_{i}'$};
				\draw		(-5.1,9.1)	..controls	(-5,9.2)	and	(-4.7,10)..	(-4.7,10.5);					
%				\draw		(2,9)	node	{$D_i$};
%				\draw		(-2,9)	node	{$D_i'$};
			\end{tikzpicture}
			\caption{\small{An example of the hexagon on $X_i$ when $\gamma = \gamma'$ and $\epsilon_0 < e^{-1} \ln(1+\sqrt{2})$, for $i=1,2$. }}
		\end{figure}

It's easy to show that the ratio of $\ell_1(\beta)$ and $\ell_2(\beta)$ satisfies
		\begin{equation}\label{eq:C2_main}
			\begin{split}
				\cfrac{\ell_1(\beta)}{\ell_2(\beta)}
				&= \cfrac{2(b_1+d_1)}{2(b_2+d_2)}
				\leq \max \{ \cfrac{b_1}{b_2},\ \cfrac{d_1}{d_2} \} \ .
%				\leq 2\ \max \{ \cfrac{b_1}{b_2},\ \cfrac{d_1}{d_2} \} \ .
			\end{split}
		\end{equation}				
		As in the case where $\gamma \neq \gamma'$, we will control $b_1/b_2$ and $d_1/d_2$  by
the  ratios of
lengths of $\alpha$ and $\gamma$.
We will prove these results in Lemma \ref{lemma:C2_Pb_Final} and Lemma \ref{lemma:C2_Pd_Final}.

We first discuss the $b_1/b_2$ part.
		\begin{lemma}\label{lemma:C2_Pb_Final}
			There exists a positive constant $K_2'$ depending on $\epsilon_0$ such that
			\begin{equation}\label{eq:C2_Pb_Final}
				\cfrac{b_1}{b_2} \leq K_2' \cdot \max\{1, \cfrac{\ell_1(\alpha)}{\ell_2(\alpha)}\} \ .
			\end{equation}
		\end{lemma}
		
		\begin{proof}
We follow the same outline as in the proof of Lemma \ref{lemma:C1_Pb_Final}.
By our assumption (see Remark \ref{remark:alpha}), $\ell_{X_1}(\alpha)\geq \ell_{X_1}(\alpha')$.
Let us first consider the case where $\ell_{X_2}(\alpha)\geq \ell_{X_2}(\alpha')$.

As $\beta_i^Q$, $i=1,2$, can be viewed as the middle part of $\beta$, by the same proof as that of \eqref{eq:C1_Pb_1st}
in Lemma \ref{lemma:bi}, we have
			\begin{equation}\label{eq:C2_Pb_2nd}
				b_i \geq \frac{M_0}{2},\text{ for } i=1,2.
			\end{equation}

			Next we  discuss the relation between $c_i'$ and $\ell_i(\gamma)$, for $i = 1, 2$.
           It is obvious that $c'_1 < \ell_1(\gamma)/2$.
			
By \eqref{eq:RightAngledPentagon}, we have
$$
\cosh a_{1}/\sinh c_1' = \sinh(b_1+d_1)= \cosh a_{1}'/\sinh c_1''.
$$
Since (by assumption) $a_{1} \geq a_{1}'$, we have
$$
\sinh c_1'/\sinh c_1''= \cosh a_{1}/\cosh a_{1}' \geq 1.
$$
			Therefore $c_1' \geq c_1''$.
			Since $2 (c_1' + c_1'') = \ell_1(\gamma)$, we have $c_1' \geq \ell_1(\gamma)/4$. We have similar result for $c'_2$.
			
It follows that
			\begin{equation}\label{eq:C2_Pb_1st}
				\cfrac{1}{4}\ \ell_i(\gamma) \leq c'_i < \cfrac{1}{2}\ \ell_i(\gamma),  \ \mathrm{ for} \ i=1,2.
			\end{equation}				
			
Since $c'_i, \ i=1,2$ are bounded above by $\frac{\epsilon_0}{2}$,
 there is a positive constant $k_2$ depending on $\epsilon_0$ such that
			\begin{equation}\label{eq:C2_Pb_4th}
				c'_i \leq \sinh c'_i \leq k_2 \cdot c'_i,  \ \mathrm{ for} \ i=1,2.
			\end{equation}

Since $b_i+d_i, \ i=1,2$ are bounded by $M_0$ from below, we can choose  $k_2$ such that
			\begin{equation}\label{eq:C2_Pb_3rd}
					k_2^{-1}\cdot e^{b_i+d_i} \leq \sinh(b_i+d_i) \leq \cfrac{1}{2} \cdot  e^{b_i+d_i}, \ \mathrm{ for} \ i=1,2.
			\end{equation}

			Similar to the case where $\gamma \neq \gamma'$, we  can estimate the difference between $a_{i}$ and $b_i$,  $i = 1 , 2$.
	
By \eqref{eq:RightAngledPentagon}, \eqref{eq:C2_Pb_1st}, \eqref{eq:C2_Pb_4th}, \eqref{eq:C2_Pb_3rd}  and the fact that $\ell_i(\gamma) \cdot \cosh d_i = \epsilon_0'$, $i=1,2$,  we have (for $i=1,2$)
			\begin{equation*}
				\begin{split}
					e^{a_{i}} &\geq \cosh a_{i}
					= \sinh(b_i+d_i) \cdot \sinh c'_i \\
					&\geq k_2^{-1} \cdot e^{b_i+d_i} \cdot c'_i
					\geq k_2^{-1} e^{b_i} \cdot \cosh{d_i} \cdot \cfrac{1}{4}\  \ell_i(\gamma) \\
					&= \cfrac{k_2^{-1} \epsilon_0'}{4} \cdot e^{b_i}
				\end{split}
			\end{equation*}	
			and
			\begin{equation*}
				\begin{split}
					e^{a_{i}} &\leq 2\ \cosh a_{i}
					= 2\ \sinh(b_i+d_i) \cdot \sinh c'_i \\
					&\leq e^{b_i+d_i} \cdot k_2\cdot c'_i
					< k_2 e^{b_i} \cdot 2\ \cosh d_i \cdot \cfrac{1}{2}\ \ell_i(\gamma) \\
					&= k_2 \epsilon_0' \cdot e^{b_i}.
				\end{split}
			\end{equation*}	
			
Let $D_2 = \max \{ |\ln (k_2^{-1} \epsilon_0') - \ln 4|,\ |\ln(k_2 \epsilon_0')| \}$.
            We have
			\begin{equation}\label{eq:C2_Pb_5th}
				|b_i - a_{i}| \leq  D_2, \  i=1,2.
			\end{equation}		

Now we have inequalities $(\ref{eq:C2_Pb_2nd})$ and $(\ref{eq:C2_Pb_5th})$,
the analogy of $(\ref{eq:C1_Pb_1st})$ and $(\ref{eq:C1_Pb_5nd})$ previously.
By the same proof as in Lemma \ref{lemma:C1_Pb_Final} (see
the discussion after Lemma \ref{lemma:difference}), we can show that there is a constant $K_2'$ depending on $\epsilon_0$ such that
$$\frac{b_1}{b_2}\leq K_2' \cdot \max \{1, \frac{a_1}{a_2}\}.$$ Since
$\frac{a_1}{a_2}=\frac{\ell_1(\alpha)}{\ell_2(\alpha)}$, we finish the proof
under the assumption that $\ell_{X_2}(\alpha)\geq \ell_{X_2}(\alpha')$.

If
$ \ell_{X_2}(\alpha)< \ell_{X_2}(\alpha'),$ then we can
modify the above argument to show that
\begin{equation*}
				\cfrac{b_1}{b_2} \leq K_2' \cdot \max\{1, \cfrac{\ell_1(\alpha)}{\ell_2(\alpha')}\} \ .
			\end{equation*}
Since $\cfrac{\ell_1(\alpha)}{\ell_2(\alpha')}\leq \cfrac{\ell_1(\alpha)}{\ell_2(\alpha)}$,
the inequality $(\ref{eq:C2_Pb_Final})$ remains true.
\end{proof}
           	Next we will discuss the $d_1/d_2$ part.
		\begin{lemma}\label{lemma:C2_Pd_Final}
			We have
			\begin{equation}\label{eq:C2_Pd_Final}
				\cfrac{d_1}{d_2}
				\leq 2 \cdot \max\{1, \cfrac{\ell_2(\gamma)}{\ell_1(\gamma)}\}.
			\end{equation}
		\end{lemma}
		
		\begin{proof}
			The proof is the same as that of Lemma \ref{lemma:C1_Pd_Final}. We have
            \begin{equation*}
                \ln \epsilon_0' - \ln \ell_i(\gamma)
                \leq d_i
                \leq 2 (\ln \epsilon_0' - \ln \ell_i(\gamma)), i=1,2.
            \end{equation*}
            Then
			\begin{equation*}
			\cfrac{d_1}{d_2}
%			\leq \cfrac{2 (\ln \epsilon_0' - \ln \ell_1(\gamma))}{\ln \epsilon_0' - \ln \ell_2(\gamma)}
			\leq 2 \cdot \cfrac{\ln \epsilon_0' - \ln \ell_1(\gamma)}{\ln \epsilon_0' - \ln \ell_2(\gamma)}.
			\end{equation*}			
			
As we did in the case where $\gamma \neq \gamma'$,
we consider the two cases depending on whether $\ell_2(\gamma) \leq \ell_1(\gamma)$ or not.
			
			If $\ell_2(\gamma) \leq \ell_1(\gamma)$, we have
 $$(\ln \epsilon_0' - \ln \ell_1(\gamma))/(\ln \epsilon_0' - \ln \ell_2(\gamma)) \leq 1.$$
			
If $\ell_2(\gamma) > \ell_1(\gamma)$,
		by the same proof as that of Lemma \ref{lemma:C1_Pd_Final}, we have
$$
(\ln \epsilon_0' - \ln \ell_1(\gamma))/(\ln \epsilon_0' - \ln \ell_2(\gamma)) \leq \ell_2(\gamma)/\ell_1(\gamma).
$$

			From the above discussions, we have
			\begin{equation*}
				\cfrac{d_1}{d_2}
				\leq 2 \cdot \max\{1, \cfrac{\ell_2(\gamma)}{\ell_1(\gamma)}\}.
			\end{equation*}

		\end{proof}

		\begin{proposition}\label{lemma:Same_Boundary}
			Let $\epsilon_0$ be a positive number with $\epsilon_0 < e^{-1} \ln(1+\sqrt{2})$ and let $X_1,X_2$ be any
two  hyperbolic metrics in the
$\epsilon_0$-relative part of $\mathcal{T}(S)$.
            For  any essential arc $\beta \in \mathcal{B}(S)$ with endpoints lying on the same boundary component $\gamma$ of $S$,
            let $\alpha$ and $\alpha'$ be the associated simple closed curves homotopic the boundaries of a regular  neighborhood of
            $\beta \cup \gamma$.
			Then there exists a positive constant $K_2$ depending on $\epsilon_0$ such that the inequality \eqref{eq:Same_Boundary} holds.
		\end{proposition}

		\begin{proof}
			By \eqref{eq:C2_main}, \eqref{eq:C2_Pb_Final} and \eqref{eq:C2_Pd_Final}, we have (with the
assumption that $\ell_{X_1}(\alpha)\geq \ell_{X_1}(\alpha')$)
			\begin{equation*}
				\begin{split}
					\cfrac{\ell_1(\beta)}{\ell_2(\beta)}
					&\leq \max\{ \cfrac{b_1}{b_2}, \cfrac{d_1}{d_2} \} \\
					&\leq \max\{ K_2' \cdot \max\{ 1, \ \cfrac{\ell_1(\alpha_1)}{\ell_2(\alpha)} \}, \ 2 \cdot \max\{1,\ \cfrac{\ell_2(\gamma)}{\ell_1(\gamma)} \} \} \\
					&\leq K_2 \cdot \max\{ 1, \  \cfrac{\ell_1(\alpha)}{\ell_2(\alpha)}, \ \cfrac{\ell_2(\gamma)}{\ell_1(\gamma)} \} \ ,
				\end{split}
			\end{equation*}
			where $K_2 = \max\{K_2',\ 2\}$ is a positive constant depending on $\epsilon_0$.
		\end{proof}
		
\subsection{Corollary.}		By Proposition \ref{lemma:Diff_Boundary} and Proposition \ref{lemma:Same_Boundary}, we have the following corollary.
%for any arc $\beta \in \mathcal{B}(S)$, we can replace its ratio by the ratio of some loops if $\epsilon_0 < e^{-1} \ln(1+\sqrt{2})$.
%		And the Lemma \ref{lemma:key} is given immediately.
		
		\begin{corollary}\label{lemma:key_1}
			Let $\epsilon_0$ be a positive number with $\epsilon_0 < e^{-1} \ln(1+\sqrt{2})$.
			There exists a positive constant $K$ depending on $\epsilon_0$ such that
			\begin{equation*}
				\begin{split}
					 \sup_{\beta \in \mathcal{B}(S)} \{ \cfrac{\ell_{X_1}(\beta)}{\ell_{X_2}(\beta)}, \cfrac{\ell_{X_2}(\beta)}{\ell_{X_1}(\beta)} \}
					&\leq K \cdot \sup_{\alpha \in \mathcal{C}(S)} \{ \cfrac{\ell_{X_1}(\alpha)}{\ell_{X_2}(\alpha)}, \cfrac{\ell_{X_2}(\alpha)}{\ell_{X_1}(\alpha)} \},
				\end{split}
			\end{equation*}		
for any $X_1,X_2$ on the $\epsilon_0$-relative part of $\mathcal{T}(S)$.	
		\end{corollary}

		\section{Proof of Theorem \ref{lemma:key}: The general case}\label{ProofTheorem2}

We have shown in Section \ref{ProofTheorem1} that the supremum of the ratio of lengths of arcs is controlled by that of simple closed curves in the case where $\epsilon_0 < e^{-1} \ln(1+\sqrt{2})$. In this section, we will prove the result in the general case.
We assume that $\epsilon_0 \geq e^{-1} \ln(1+\sqrt{2})$.

		Here is the idea of the proof. Recall that
in Section \ref{ProofTheorem1}, we separated the arc $\beta$ into several parts.
		In the case where $\epsilon_0 < e^{-1} \ln(1+\sqrt{2})$, the length of $\beta^{Q}_{i}$ is bounded below by a positive number and $\ell_{X_1}(\beta_1^Q)/\ell_{X_2}(\beta_2^Q)$ is controlled by the ratio of the lengths of some corresponding simple closed curves.
But in the case where $\epsilon_0 \geq e^{-1} \ln(1+\sqrt{2})$, in general, it's impossible to give a lower bound for $\ell_{X_2}(\beta_2^{Q})$.
        To deal with this, we will not separate the arc $\beta$ into parts unless the width of collar neighborhood of
        $\gamma$ or $\gamma'$ is large enough.

Let $\epsilon_0' = \ln (1+\sqrt{2})$.
		Let $X_1$ and $X_2$ be in the $\epsilon_0$-relative part of $\mathcal{T}(S)$.
		We apply the same notations $\beta$, $\gamma$, $\gamma'$ and $\alpha$ as in Section \ref{ProofTheorem1}.

\subsection{The case where $\gamma \neq \gamma'$.}

		First we will consider the case where $\gamma \neq \gamma'$.
        We define $C_i$ and $C_i'$, $i=1,2$,  as follows.

        If $\ell_{X_1}(\gamma) < e^{-1} \ln (1 + \sqrt{2})$,
        let $C_1$ be a closed curve isotopic to $\gamma$ with $\ell_{X_1}(C_1) = \epsilon_0'$ such that $C_1$ and $\gamma$ are the boundaries of a regular annulus around $\gamma$ on $X_1$.
	Otherwise, we let $C_1 = \gamma$. Similarly, we can define $C_1'$.  The corresponding notations on $X_2$ are obtained only by replacing subscript.
		
We have four cases according to whether $\ell_{X_i}(\gamma)$ or $\ell_{X_i}(\gamma')$, $i=1,2$,  is less than $e^{-1} \ln(1+\sqrt{2})$ or not.
		Figure 9  shows how to choose $C_i$ and $C_i'$ in each case, for $i =1,2$.
		
		\begin{figure}[htb]\label{figure:9}
			\subfigure [Case $\ell_{i}(\gamma) < e^{-1} \ln(1+\sqrt{2})$ and $\ell_{i}(\gamma') < e^{-1} \ln(1+\sqrt{2})$]
			{
				\begin{tikzpicture}[scale=0.5]
					\draw		(-9.6,0)	node	{$\gamma$};
					\draw		(-9,0.5)	arc	(90:-90:	0.25	and 0.5);
					\draw		(-9,0.5)	arc	(90:270:	0.25	and 0.5);
					\draw		(-4.5,2)	node	{$C_i$};
					\draw		(-4.5,1.5)	arc	(90:-90:	0.5	and 1);
					\draw[dashed]	(-4.5,1.5)	arc	(90:270:	0.5	and 1);
					\draw		(-9,0.5)	..controls	(-8,0.5)	and	(-6,1)..	(-4.5,1.51);
	
					\draw		(9.8,0)	node	{$\gamma'$};
					\draw		(9,0.5)	arc	(90:-90:	0.25	and 0.5);
					\draw	[dashed](9,0.5)	arc	(90:270:	0.25	and 0.5);
					\draw		(4.5,2)	node	{$C_i'$};
					\draw		(4.5,1.5)	arc	(90:-90:	0.5	and 1);
					\draw[dashed]	(4.5,1.5)	arc	(90:270:	0.5	and 1);
					\draw		(9,0.5)	..controls	(8,0.5)	and	(6,1)..	(4.5,1.51);
					
					\draw		(-9,-0.5)	..controls	(-4.5,-0.5)	and	(4.5,-0.5)..	(9,-0.5);
	
					\draw		(0,4)	node	{$\alpha$};
					\draw		(-3,4)	arc	(180:360:	3	and	1);
					\draw[dashed](-3,4)	arc	(180:0:	3	and	1);
					
					\draw		(-4.5,1.51)	..controls	(-4,1.8)	and	(-3,2)..	(-3,4);
					\draw		(-3.5,5)	..controls	(-3.5,5)	and	(-3,5)..	(-3,4);
					\draw		(4.5,1.51)	..controls	(4,1.8)		and	(3,2)..	(3,4);
					\draw		(3.5,5)	..controls	(3.5,5)	and	(3,5)..	(3,4);
					
					\draw		(0,-1.2)	node	{$\beta$};
					\draw		(0,0.9)	node	{$\beta^{Q}_{i}$};
					\draw		(-6,0.6)	node	{$\beta^{A}_{i}$};
					\draw		(6,0.6)	node	{${\beta'}^{A}_{i}$};
					\draw		(-8.75,0)..controls	(-4.5,0.4)	and	(4.5,0.4)..	(9.25,0);
				\end{tikzpicture}
			}

			\subfigure [Case $\ell_{i}(\gamma) \geq e^{-1} \ln(1+\sqrt{2})$ and $\ell_{i}(\gamma') < e^{-1} \ln(1+\sqrt{2})$]
			{
				\begin{tikzpicture}[scale=0.5]
					\draw		(-6,0.5)	node	{$C_i = \gamma$};
					\draw		(-4.3,1.5)	arc	(90:-90:	0.5	and 1);
					\draw		(-4.3,1.5)	arc	(90:270:	0.5	and 1);
	
					\draw		(9.8,0)	node	{$\gamma'$};
					\draw		(9,0.5)	arc	(90:-90:	0.25	and 0.5);
					\draw	[dashed](9,0.5)	arc	(90:270:	0.25	and 0.5);
					\draw		(4.5,2)	node	{$C_i'$};
					\draw		(4.5,1.5)	arc	(90:-90:	0.5	and 1);
					\draw[dashed]	(4.5,1.5)	arc	(90:270:	0.5	and 1);			
					\draw		(9,0.5)	..controls	(8,0.5)	and	(6,1)..	(4.5,1.5);
						
					\draw		(-4.3,-0.5)	..controls	(-4.5,-0.5)	and	(4.5,-0.5)..	(9,-0.5);
	
					\draw		(0,4)	node	{$\alpha$};
					\draw		(-3,4)	arc	(180:360:	3	and	1);
					\draw[dashed](-3,4)	arc	(180:0:	3	and	1);
					
					\draw		(-4.4,1.47)	..controls	(-4,1.6)	and	(-3,2)..	(-3,4);
					\draw		(-3.5,5)	..controls	(-3.5,5)	and	(-3,5)..	(-3,4);
					\draw		(4.5,1.5)	..controls	(4,1.8)		and	(3,2)..	(3,4);
					\draw		(3.5,5)	..controls	(3.5,5)	and	(3,5)..	(3,4);
					
					\draw		(2,-1.2)	node	{$\beta$};
					\draw		(0,0.9)	node	{$\beta^{Q}_{i}$};
					\draw		(6,0.6)	node	{${\beta'}^{A}_{i}$};
					
					\draw		(-3.8,0.4)..controls	(0,0.4)	and	(4.5,0.4)..	(9.25,0);
				\end{tikzpicture}
			}
				
			\subfigure [Case $\ell_{i}(\gamma) < e^{-1} \ln(1+\sqrt{2})$ and $\ell_{i}(\gamma') \geq e^{-1} \ln(1+\sqrt{2})$]
			{
				\begin{tikzpicture}[scale=0.5]
					\draw		(6.2,0.5)	node	{$C_i' = \gamma'$};
					\draw		(4.3,1.5)	arc	(90:-90:	0.5	and 1);
					\draw[dashed]	(4.3,1.5)	arc	(90:270:	0.5	and 1);
	
					\draw		(-9.6,0)	node	{$\gamma$};
					\draw		(-9,0.5)	arc	(90:-90:	0.25	and 0.5);
					\draw		(-9,0.5)	arc	(90:270:	0.25	and 0.5);
					\draw		(-4.5,2)	node	{$C_i$};
					\draw		(-4.5,1.5)	arc	(90:-90:	0.5	and 1);
					\draw[dashed]	(-4.5,1.5)	arc	(90:270:	0.5	and 1);			
					\draw		(-9,0.5)	..controls	(-8,0.5)	and	(-6,1)..	(-4.5,1.5);
						
					\draw		(4.3,-0.5)	..controls	(4.5,-0.5)	and	(-4.5,-0.5)..	(-9,-0.5);
	
					\draw		(0,4)	node	{$\alpha$};
					\draw		(-3,4)	arc	(180:360:	3	and	1);
					\draw[dashed](-3,4)	arc	(180:0:	3	and	1);
					
					\draw		(4.4,1.47)	..controls	(4,1.6)	and	(3,2)..	(3,4);
					\draw		(3.5,5)	..controls	(3.5,5)	and	(3,5)..	(3,4);
					\draw		(-4.5,1.5)	..controls	(-4,1.8)		and	(-3,2)..	(-3,4);
					\draw		(-3.5,5)	..controls	(-3.5,5)	and	(-3,5)..	(-3,4);
					
					\draw		(-2,-1.2)	node	{$\beta$};
					\draw		(-0,0.9)	node	{$\beta^{Q}_{i}$};
					\draw		(-6,0.6)	node	{$\beta^{A}_{i}$};
					
					\draw		(4.8,0.4)..controls	(0,0.4)	and	(-4.5,0.4)..	(-8.75,0);
				\end{tikzpicture}
			}
				
			\subfigure [Case $\ell_{i}(\gamma) \geq e^{-1} \ln(1+\sqrt{2})$ and $\ell_{i}(\gamma') \geq e^{-1} \ln(1+\sqrt{2})$]
			{
				\begin{tikzpicture}[scale=0.5]
					\draw		(6.5,0.5)	node	{$C_i' = \gamma'$};
					\draw		(4.3,1.5)	arc	(90:-90:	0.5	and 1);
					\draw[dashed]	(4.3,1.5)	arc	(90:270:	0.5	and 1);
	
					\draw		(-6.5,0.5)	node	{$C_i = \gamma$};
					\draw		(-4.3,1.5)	arc	(90:-90:	0.5	and 1);
					\draw		(-4.3,1.5)	arc	(90:270:	0.5	and 1);
	
					\draw		(4.3,-0.5)	..controls	(4.5,-0.5)	and	(-4.5,-0.5)..	(-4.3,-0.5);
					
					\draw		(0,4)	node	{$\alpha$};
					\draw		(-3,4)	arc	(180:360:	3	and	1);
					\draw[dashed](-3,4)	arc	(180:0:	3	and	1);
					
					\draw		(4.4,1.47)	..controls	(4,1.6)	and	(3,2)..	(3,4);
					\draw		(3.5,5)	..controls	(3.5,5)	and	(3,5)..	(3,4);
					\draw		(-4.4,1.47)	..controls	(-4,1.6)	and	(-3,2)..	(-3,4);
					\draw		(-3.5,5)	..controls	(-3.5,5)	and	(-3,5)..	(-3,4);
					
%					\draw		(0,-1.2)	node	{$\beta$};
					\draw		(-0,1)	node	{$\beta^{Q}_{i} = \beta$};

					\draw		(4.8,0.4)..controls	(3,0.5)	and	(-3,0.5)..	(-3.8,0.4);
				\end{tikzpicture}
			}

			\caption{\small{Examples of pair of pants on $X_i$ when $\gamma \neq \gamma'$ and $\epsilon_0 \geq e^{-1} \ln(1+\sqrt{2})$, for $i = 1, 2$. }}
		\end{figure}
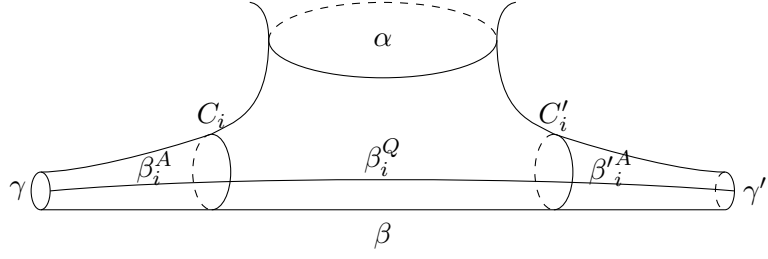
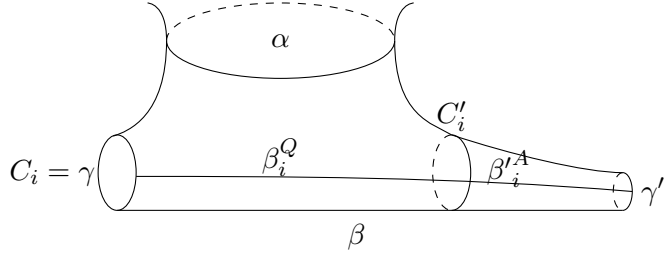
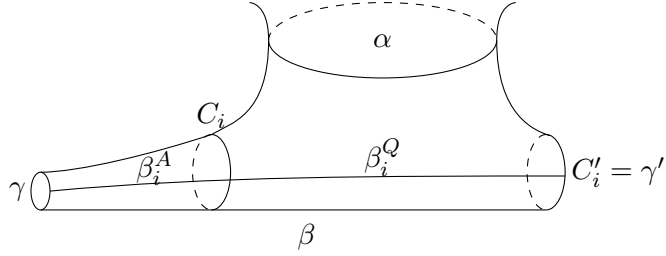
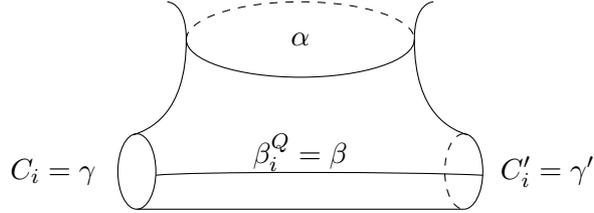
				
			Again, to simplify  notation, we denote $\ell_i=\ell_{X_i}$, for $i = 1,2$.
			Let $a_i = \ell_i(\alpha)/2$, $b_i = \ell_i(\beta)/2$, $c_i = \ell_i(\gamma)/2$, $c'_i = \ell_i(\gamma')/2$, $d_i = \ell_i(\beta^{A}_{i})$, $d'_i = \ell_i({\beta'}^{A}_{i})$, $b'_i = \ell_i(\beta^{Q}_{i})$, $e_i = \ell_i(C_i)/2$, $e_i' = \ell_i(C_i')/2$ and $c_0' = \epsilon_0'/2$, for $i = 1, 2$.
It should be noted that $d_i = 0$ and $e_i = c_i$ if $\ell_i(\gamma) \geq e^{-1} \ln(1+\sqrt{2})$.

\bigskip

With the above notations, let us describe Figure 9 in more details:
			
\bigskip

			Case (a): $\ell_i(\gamma)   <  e^{-1} \ln(1+\sqrt{2})$ and $\ell_i(\gamma')   <  e^{-1} \ln(1+\sqrt{2})$, $i=1,2$.

            \bigskip

            In this case, $e_i = e_i' = c_0'$, $i=1,2$.
            This implies that $d_i$ and $d_i'$, $i=1,2$ are positive.
			
 \bigskip

			Case (b): $\ell_i(\gamma) \geq e^{-1} \ln(1+\sqrt{2})$ and $\ell_i(\gamma')   <  e^{-1} \ln(1+\sqrt{2})$, $i=1,2$.

            \bigskip

             In this case, $e_i = c_i$ and $e_i' = c_0'$, $i=1,2$.
            This implies that $d_i = 0$ and $d_i > 0$, $i=1,2$.
			
 \bigskip

			Case (c): $\ell_i(\gamma)   <  e^{-1} \ln(1+\sqrt{2})$ and $\ell_i(\gamma') \geq e^{-1} \ln(1+\sqrt{2})$, $i=1,2$.

             \bigskip

            Similar to Case (b), $e_i = c_0'$ and $e_i' = c_i'$, $i=1,2$.
            This implies that $d_i > 0$ and $d_i' = 0$, $i=1,2$.
			
			 \bigskip

Case (d): $\ell_i(\gamma) \geq e^{-1} \ln(1+\sqrt{2})$ and $\ell_i(\gamma') \geq e^{-1} \ln(1+\sqrt{2})$, $i=1,2$.

            \bigskip

             In this case, we have $e_i = c_i$, $e_i' = c_i'$ and $d_i = d_i' = 0$, $i=1,2$.
			
\bigskip

			These four cases are shown in Figure 9.
The ratio of the lengths of $\beta$ on $X_1$ and $X_2$ satisfies the following
			\begin{equation*}
				\cfrac{\ell_1(\beta)}{\ell_2(\beta)} = \cfrac{b_1}{b_2} = \cfrac{b_1'+d_1+d_1'}{b_2'+d_2+d_2'}
				\leq 3 \cdot \max\{\cfrac{b_1'}{b_2'},\ \cfrac{d_1}{b_2'+d_2},\ \cfrac{d_1'}{b_2'+d_2'}\} \ .
			\end{equation*}
			We will study each part on the right hand side of the above inequality.

			\begin{lemma}\label{lemma:C1_Pb_Final_2}
				There exists a positive constant $K_3'$ depending on $\epsilon_0$ such that
				\begin{equation}\label{eq:C1_Pb_Final_2}
					\cfrac{b_1'}{b_2'} \leq K_3' \cdot \max\{1,\cfrac{a_1}{a_2}\} \ .
				\end{equation}
			\end{lemma}

			\begin{proof}
				We first show that there exists a positive lower bound (that may depend on $\epsilon_0$) for $b_i'$, for $i = 1,2$.
				Consider $b_1'$. We have to consider all four cases as illustrated in Figure 9.
						
				In Case (a),  by \eqref{eq:C1_Pb_1st},
					we have $b_1' \geq 8/\ln(1+\sqrt{2})$.
				
In Case (b) or Case (c),
					the proof of \eqref{eq:C1_Pb_1st} shown that $b_1' \geq 4/\ln(1+\sqrt{2})$.

In Case (d), $C = \gamma$ and $C' = \gamma'$.
					By \eqref{eq:RightAngledHexagon}, we have $$\cosh a_1 + \cosh c_1 \cosh c_1' = \sinh c_1 \sinh c_1' \cosh b_1'.$$
					Since $c_1 \leq \epsilon_0/2$ and $c_1' \leq \epsilon_0/2$, we have a lower bound for  $\cosh b_1'$:
					\begin{equation*}
						\begin{split}
							\cosh b_1' &= \cfrac{\cosh a_1 + \cosh c_1 \cosh c_1'}{\sinh c_1 \sinh c_1'}\\
							&= \cfrac{\cosh a_1}{\sinh c_1 \sinh c_1'} + \coth c_1 \coth c_1'
							\geq ({\sinh \cfrac{\epsilon_0}{2}})^{-2} + 1 \ .
						\end{split}
					\end{equation*}
                    It follows that $b_1' \geq \text{arcosh} \left((\sinh (\epsilon_0 /2)\right)^{-2}+1)$.
                   Using the same argument we have that $b_2' \geq \text{arcosh} ((\sinh (\epsilon_0 /2))^{-2}+1)$.
		Let $$M' = \min\{ 4/\ln(1+\sqrt{2}), \ \text{arcosh} ((\sinh (\epsilon_0 /2))^{-2}+1)\},$$ then
              we have
				\begin{equation}\label{eq:C1_Pb_1st_2}
					b_i' \geq M',\text{ for }i=1,2.
				\end{equation}

				Now we claim that the difference between $a_i$ and $b_i'$ is bounded from above in all the above four cases, for $i = 1, 2$.
                We only give the discussion on $X_1$. The discussion on $X_2$ is the same.
				
				Case (a) is handled by inequality \eqref{eq:C1_Pb_5nd}.

				In Case (b) or Case (c),	
					it is sufficient to consider Case (b). The discussion of Case $(c)$ works in the same way.
					As $\ell_1(\gamma) \geq e^{-1} \ln(1+\sqrt{2})$ and $d_1=0$, we have
					\begin{equation*}
						\begin{split}
							\sinh c_1 \sinh c_1' \cosh ( b_1'+d_1+d_1')
							&> \cfrac{1}{2} \cdot c_1 \cdot c_1' \cosh d_1' \cdot e^{b_1'}
							= \cfrac{1}{2} \cdot \cfrac{\ell_1(\gamma)}{2} \cdot c_0' \cdot e^{b_1'} \\
                            &\geq \cfrac{(\ln(1+\sqrt{2}))^2}{8 e} \cdot e^{b_1'} \ .
						\end{split}
					\end{equation*}
					As $c_1 < \epsilon_0/2$, $c_1' < \epsilon_0/2$, we have a positive constant $k_1$ (depending on $\epsilon_0$) that satisfies $\sinh c_1 \leq k_1 \cdot c_1$ and $\sinh c_1' \leq k_1 \cdot c_1'$.
					Then we have
					\begin{equation*}
						\begin{split}
							\sinh c_1 \sinh c_1' \cosh ( b_1'+d_1+d_1')
							&< k_1^2 \cdot c_1 \cdot c_1'\ \cosh d_1' \cdot e^{b_1'}\\
                            &\leq \cfrac{\epsilon_0 k_1^2 \ln(1+\sqrt{2})}{4 e} \cdot e^{b_1'} \ .
						\end{split}
					\end{equation*}
					Let $M_1 = \max \{ 8e/(\ln(1+\sqrt{2}))^2, \ \epsilon_0 k_1^2 \ln(1+\sqrt{2}) /(4e) \}$. Then we have proved that
 $$M_1^{-1} e^{b_1'} \leq \sinh c_1 \sinh c_1' \cosh (b_1'+d_1+d_1) \leq M_1 e^{b_1'}.$$
					Similarly to the proof of inequality \eqref{eq:C1_Pb_5nd}, we can show that there exists a positive constant $D_2$ depending on $\epsilon_0$ such that $|a_1 - b_1'| \leq D_2$.

				In Case (d), the assumption of this case implies that $d_1=0$ and $d_1'=0$.
					It follows that
					\begin{equation*}
						\begin{split}
							\sinh c_1 \sinh c_1' \cosh (b_1'+d_1+d_1')
							&> \cfrac{c_1 c_1'}{2} \cdot e^{b_1'}
                            \geq \cfrac{(\ln(1+\sqrt{2}))^2}{8 e^2} \cdot e^{b_1'}
						\end{split}
					\end{equation*}
					and
					\begin{equation*}
						\begin{split}
							\sinh c_1 \sinh c_1' \cosh (b_1'+d_1+d_1')
							&< K_1^2 c_1 c_1' e^{b_1'}
                            \leq \cfrac{K_1^2 \epsilon_0^2}{4} e^{b_1'}.
						\end{split}
					\end{equation*}
					Let $M_1 = \max \{ 8 e^2 / (\ln(1+\sqrt{2}))^2 ,\ K_1^2 \epsilon_0^2 / 4\}$.
					Again, similarly to the proof of inequality \eqref{eq:C1_Pb_5nd},  we can prove that  $|a_1 - b_1'| \leq D_3$, where the constant $D_3$ depends on
$\epsilon_0$.	
				
The same proof applies to $|a_2-b_2'|$. Then we have a positive constant $D$ depending on $\epsilon_0$ such that  				
				\begin{equation}\label{eq:C1_Pb_2nd_2}
					|a_i-b_i'| \leq D, i=1,2.
				\end{equation}

 Comparing $(\ref{eq:C1_Pb_1st_2})$ and $(\ref{eq:C1_Pb_2nd_2})$ with  $(\ref{eq:C1_Pb_1st})$ and $(\ref{eq:C1_Pb_5nd})$ in the proof of Lemma \ref{lemma:C1_Pb_Final}, we finish the proof.
			\end{proof}			

			Next we will consider the $d_1/(b_2'+d_2)$ part.
			\begin{lemma}\label{lemma:C1_Pd_Final_2}
				There exists a positive constant $K_3''$ depending on $\epsilon_0$ such that
				\begin{equation}\label{eq:C1_Pd_Final_2}
					\cfrac{d_1}{b_2'+d_2} \leq K_3'' \cdot \max\{1,\ \cfrac{c_2}{c_1}\} \ .
				\end{equation}
			\end{lemma}
			
			\begin{proof}

By $\eqref{eq:RegularAnnulis}$, we have
				\begin{equation*}\label{eq:C1_Pd_1st_2}
					\ell_i(\gamma) \cdot \cosh d_i = \ell_i(C_i), \  i =1, 2.
				\end{equation*}
                Similarly, $$\ell_i(\gamma') \cdot \cosh d_i' = \ell_i(C_i'), \ i =1, 2.$$
						
	 			We consider the two cases depending on whether $\ell_1(\gamma) < e^{-1} \ln(1+\sqrt{2})$ or not.
	 			
\bigskip
	 			\begin{case}{1: $\ell_1(\gamma) < e^{-1} \ln(1+\sqrt{2})$. }
	 \bigskip			
                    By assumption, we have $\ell_1(C_1) = \ln (1+\sqrt{2})$.
                   As in the proof of Lemma \ref{lemma:C1_Pd_Final}, we have
                    \begin{equation}\label{eq:4-2}
                        \ln \ell_1(C_1) - \ln \ell_1(\gamma)
                        \leq d_1
                        \leq  2(\ln \ell_1(C_1) - \ln \ell_1(\gamma)).
                    \end{equation}
	 				If $\ell_2(\gamma) < e^{-1} \ln(1+\sqrt{2})$, the lemma follows from Lemma \ref{lemma:C1_Pd_Final}.
	 				If  $\ell_2(\gamma) \geq e^{-1} \ln(1+\sqrt{2})$, then  $d_2 = 0$.
	 					By \eqref{eq:4-2}, we have
	 					\begin{equation*}
	 						\begin{split}
	 							\cfrac{d_1}{b_2'+d_2} &\leq \cfrac{d_1}{M'}
	 							\leq \cfrac{2}{M'} \cdot \ln \cfrac{\ell_1(C_1)}{\ell_1(\gamma)}
	 							\leq \cfrac{2}{M'} \cdot \cfrac{\ell_1(C_1)}{\ell_1(\gamma)} \\
	 							&\leq \cfrac{2 \ln(1+\sqrt{2})}{M' \ell_2(\gamma)} \cdot \cfrac{\ell_2(\gamma)}{\ell_1(\gamma)}
                                \leq \cfrac{2  e}{M'} \cdot \cfrac{\ell_2(\gamma)}{\ell_1(\gamma)} \ .	
	 						\end{split}
	 					\end{equation*}
	 			
	 			\end{case}
	 			\bigskip

	 			\begin{case}{2: $\ell_1(\gamma) \geq e^{-1} \ln(1+\sqrt{2})$.}
\bigskip

	 				Since $d_1 = 0$ and $b_2' \geq M'$, we have $d_1/(b_2'+d_2) = 0$.
	 			\end{case}
	 			
\bigskip
                Let $K_3'' = \max \{ 2, 2e / M'\}$. Since $\cfrac{c_2}{c_1}=\cfrac{\ell_2(\gamma)}{\ell_1(\gamma)}$, we have
                $$\cfrac{d_1}{b_2'+d_2} \leq K_3'' \cdot \max\{1,\ \cfrac{c_2}{c_1}\}.$$
	 			This completes the proof.

	 		\end{proof}

We argue similarly for
\begin{lemma}\label{lemma:C1_Pd_Final_2*}
				There exists a positive constant $K_3''$ depending on $\epsilon_0$ such that
				\begin{equation}\label{eq:C1_Pd_Final_2}
					\cfrac{d_1'}{b_2'+d_2'} \leq K_3''' \cdot \max\{1,\ \cfrac{c_2'}{c_1'}\} \ .
				\end{equation}
			\end{lemma}

			\begin{proposition}\label{lemma:Diff_Boundary_2}
              For any
essential arc $\beta \in \mathcal{B}(S)$ with endpoints lying on different boundary components $\gamma$ and $\gamma'$ of $S$, let $\alpha$ be the associated simple closed curve isotopic to the boundary of a regular neighborhood of $\beta \cup \gamma \cup \gamma'$.
				Then there exists a positive constant $K_3$ depending on $\epsilon_0$ such that the following inequality holds for any $X_1, X_2$ in the $\epsilon_0$-relative part of $\mathcal{T}(S)$:
				\begin{equation}\label{eq:Diff_Boundary_2}
					\begin{split}
						\cfrac{\ell_{X_1}(\beta)}{\ell_{X_2}(\beta)}
						&\leq K_3 \cdot \max \{ 1, \
						\cfrac{\ell_{X_1}(\alpha)}{\ell_{X_2}(\alpha)}, \
						\cfrac{\ell_{X_2}(\gamma)}{\ell_{X_1}(\gamma)}, \
						\cfrac{\ell_{X_2}(\gamma')}{\ell_{X_1}(\gamma')}\} \ .
					\end{split}
				\end{equation}
			\end{proposition}
	 		\begin{proof}
	 			By Lemma \ref{lemma:C1_Pb_Final_2}, Lemma \ref{lemma:C1_Pd_Final_2} and Lemma \ref{lemma:C1_Pd_Final_2*}, the ratio $\ell_1(\beta)/\ell_2(\beta)$ satisfies
		 		\begin{equation*}
		 			\begin{split}
		 				\cfrac{\ell_1(\beta)}{\ell_2(\beta)}
		 				&\leq 3 \cdot \max \{ \cfrac{b_1'}{b_2'}, \ \cfrac{d_1}{b_2'+d_2}, \ \cfrac{d_1'}{b_2'+d_2'}\} \\
		 				&\leq 3 \cdot \max \{K_3'\cdot\max\{1,\ \cfrac{a_1}{a_2}\},\ K_3''\cdot\max\{1,\ \cfrac{c_2}{c_1}\},\ K_3'''\cdot\max\{1,\ \cfrac{c_2'}{c_1'}\} \} \\
		 				&\leq K_3 \cdot \max \{ 1,\ \cfrac{a_1}{a_2},\ \cfrac{c_2}{c_1},\ \cfrac{c_2'}{c_1'}\} \\
		 				&= K_3 \cdot \max\{ 1,\ \cfrac{\ell_1(\alpha)}{\ell_2(\alpha)},\ \cfrac{\ell_2(\gamma)}{\ell_1(\gamma)},\ \cfrac{\ell_2(\gamma')}{\ell_1(\gamma')}\} \ ,
		 			\end{split}
		 		\end{equation*}
		 		where $K_3 = 3 \cdot \max\{K_3',\ K_3''\  K_3'''\}$.
			\end{proof}

		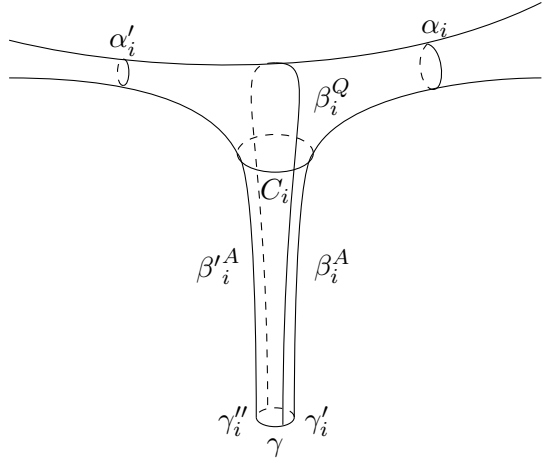
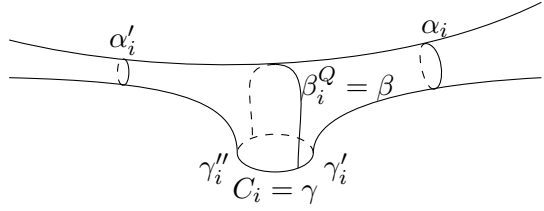
\begin{figure}[ht]\label{fig:Case_2_Surface}
			\subfigure[Case $\ell_{X_i}(\gamma) < e^{-1} \ln(1+\sqrt{2})$]
			{
				\begin{tikzpicture}[scale=0.5]
					\draw		(0,-0.8)	node	{$\gamma$};
                    \draw       (-1.1,-0.1)    node{$\gamma_i''$};
                    \draw       (1.1,-0.1)    node{$\gamma_i'$};
					\draw		(-0.5,0)	arc	(180:360:	0.5	and	0.25);
					\draw[dashed](0.5,0)	arc	(0:180:	0.5	and	0.25);				
					\draw		(0,6)	node	{$C_i$};
					\draw		(-1,7)	arc	(180:360:	1	and	0.5);
					\draw[dashed](1,7)	arc	(0:180:	1	and	0.5);
					
					\draw		(-0.5,0)	..controls	(-0.5,1)	and	(-0.5,6.5)..	(-1,7);
					\draw		(0.5,0)	..controls	(0.5,1)	and	(0.5,6.5)..	(1,7);
					
					\draw		(-1,7)	..controls	(-1.5,8)	and	(-3,9)..	(-7,9);
					\draw		(1,7)	..controls	(1.75,8.5)	and	(5,9)..	(7,9);
					\draw		(7,11)	..controls	(3,9)		and	(-3,9)..	(-7,10);
					
					\draw		(-4,10)	node	{$\alpha_i'$};
					\draw		(-4,8.8)	..controls	(-3.8,8.8)	and	(-3.8,9.5)..	(-4,9.5);
					\draw	[dashed](-4,8.8)	..controls	(-4.2,8.8)	and	(-4.2,9.5)..	(-4,9.5);
					\draw		(4.2,10.3)	node	{$\alpha_i$};
					\draw		(4.2,8.7)	..controls	(4.5,8.8)	and	(4.4,9.8)..	(4,9.9);
					\draw	[dashed](4.2,8.7)	..controls	(3.9,8.6)	and	(3.6,9.5)..	(4,9.9);
					
					\draw		(1.5,4)	node	{$\beta^{A}_{i}$};
					\draw		(-1.5,4)	node	{${\beta'}^{A}_{i}$};
					\draw		(1.5,8.5)	node	{$\beta^{Q}_{i}$};
					\draw		(0,9.4)		..controls	(1.3,9.4)		and	(0.2,8)..	(0.2,-0.2);
					\draw	[dashed](0,9.4)		..controls	(-1.3,9.4)		and	(-0.2,8)..	(-0.2,0.2);
				\end{tikzpicture}
			}
			
			\subfigure[Case $\ell_{X_i}(\gamma) \geq e^{-1} \ln(1+\sqrt{2})$]
			{
				\begin{tikzpicture}[scale=0.5]
					\draw		(0,6)	node	{$C_i = \gamma$};
                    \draw       (-1.6,6.6)    node{$\gamma_i''$};
                    \draw       (1.6,6.6)    node{$\gamma_i'$};
					\draw		(-1,7)	arc	(180:360:	1	and	0.5);
					\draw[dashed](1,7)	arc	(0:180:	1	and	0.5);

					\draw		(-1,7)	..controls	(-1,8.5)	and	(-3,8.9)..	(-7,9);
					\draw		(1,7)	..controls	(1,8.5)	and	(5,8.9)..	(7,9);
					\draw		(7,11)	..controls	(3,9)		and	(-3,9)..	(-7,10);
					
					\draw		(-4,10)	node	{$\alpha_i'$};
					\draw		(-4,8.8)	..controls	(-3.8,8.8)	and	(-3.8,9.5)..	(-4,9.5);
					\draw[dashed]	(-4,8.8)	..controls	(-4.2,8.8)	and	(-4.2,9.5)..	(-4,9.5);
					\draw		(4.2,10.3)	node	{$\alpha_i$};
					\draw		(4.2,8.7)	..controls	(4.5,8.8)	and	(4.4,9.8)..	(4,9.9);
					\draw[dashed]	(4.2,8.7)	..controls	(3.9,8.6)	and	(3.6,10)..	(4,9.9);
					
					\draw		(1.9,8.8)	node	{$\beta^{Q}_{i} = \beta$};
					\draw		(0,9.35)		..controls	(1,9.35)		and	(0.6,8)..	(0.6,6.6);
					\draw[dashed]	(0,9.35)		..controls	(-1,9.35)		and	(-0.6,8.5)..	(-0.6,7.4);
				\end{tikzpicture}
			}
			\caption{\small{Examples of pair of pants on $X_i$ when $\gamma = \gamma'$ and $\epsilon_0 \geq e^{-1} \ln(1+\sqrt{2})$. }}
		\end{figure}	

\subsection{The case where $\gamma=\gamma'$.}
$\newline$
		
%		\bigskip
		
		For hyperbolic structures $X_i$,  $i=1,2$, on $S$, we choose $\alpha_i$, $\alpha_i'$, $\gamma_i'$ and $\gamma_i''$, $i=1,2$, as  in the beginning of subsection \ref{3.3}.
        We repeat the constructions as follows.
        Let $\alpha$ and $\alpha'$ be  the boundaries of the regular neighborhood of $\beta \cup \gamma$.  We
         may assume that $\ell_i(\alpha) \geq \ell_i(\alpha')$, for $i = 1,2$.
        As we show in Figure 10, the arc $\beta$ separates $\gamma$ into two sub-arcs $\gamma_i'$ and $\gamma_i''$, such that $\gamma_i' \bigcup \beta$ is isotopic to $\alpha$, for $i = 1,2$.

        Similar to the criterion given in the above subsection, we choose  closed curves $C_i$, $i=1,2$,  which are isotopic to $\gamma$ as follows.
        Denote $\ell_{X_i}$ by $\ell_i$, $i = 1, 2$.
        If $\ell_i(\gamma) \geq e^{-1} \ln(1+\sqrt{2})$,  then we let $C_i=\gamma$, for $i=1,2$, as we shown in (b) of Figure 10.
        Otherwise, we let $C_i$ be the inner boundary of a regular annulus around $\gamma$ with $\ell_i(C_i) = \epsilon_0'$, for $i = 1,2$, as we show in (a) of Figure 10  .

       Cutting along the three geodesic arcs which connect any two of the boundaries of the pair of pants,
       we obtain  two symmetric right-angled hexagons.
        We only need consider one of them.

        If $\ell_i(\gamma) < e^{-1} \ln(1+\sqrt{2})$, the part of $\beta$ in one of the hexagon is separated by $C_i$ into two sub-arcs, for $i = 1,2$.
        As we show in Figure 10 (b), let $b_i = \ell_i(\beta_i^Q)/2$ and $d_i = \ell_i(\beta_i^A) = \ell_i({\beta_i'}^A)$, for $i = 1,2$.

        If $\ell_i(\gamma) \geq e^{-1} \ln(1+\sqrt{2})$, since $C_i = \gamma$, we let $b_i = \ell_i(\beta) / 2$ and $d_i = 0$, for $i =1,2$.
        For sake of simplicity, let $c_i' = \ell_i(\gamma_i')$, $c_i''=\ell_i(\gamma_i'')$,
         $a_i=\ell_i(\alpha)/2$, $a_{i}'=\ell_i(\alpha')/2$ and $c_0' = \epsilon_0'/2 = \ln(1+\sqrt{2})/2$, for $i=1,2$.

        In either case,  $\ell_i(\beta) = b_i + d_i$,  $i=1,2$.
		It's easy to show that
        \begin{equation*}
            \begin{split}
                \cfrac{\ell_1(\beta)}{\ell_2(\beta)}
                = \cfrac{2 \cdot (b_1+d_1)}{2 \cdot (b_2+d_2)}
                \leq 2 \cdot \max \{ \cfrac{b_1}{b_2}, \ \cfrac{d_1}{b_2+d_2} \} \ .
            \end{split}
        \end{equation*}
We will study the $b_1/b_2$ part and the $d_2/(b_2+d_2)$ part.
		
		We first consider the $b_1/b_2$ part.
		
		\begin{lemma}\label{lemma:C2_Pb_Final_2}
			There exists a positive constant $K_4'$ depending on $\epsilon_0$ such that
			\begin{equation}\label{eq:C2_Pb_Final_2}
				\cfrac{b_1}{b_2} \leq K_4' \cdot \max\{1,\ \cfrac{\ell_1(\alpha)}{\ell_2(\alpha)}\} \ .
			\end{equation}
		\end{lemma}
		
		\begin{proof}
			As in the proof of \eqref{eq:C2_Pb_1st}, we have
			\begin{equation*}
				\cfrac{\ell_i(\gamma)}{4} \leq c_i' \leq \cfrac{\ell_i(\gamma)}{2},\  i=1,\ 2\ .
			\end{equation*}
			We will show that there exists a positive lower bound for $b_i$, $i = 1,2$.

 Consider $b_1$  first. If $l_{1} (\gamma)<e^{-1} \ln (1+ \sqrt{2})$, then $b_1\geq  4/{\ln(1+\sqrt{2})}$, see \eqref{eq:C2_Pb_2nd}.
Now we suppose that $l_{1} (\gamma) \geq e^{-1} \ln (1+ \sqrt{2})$. By \eqref{eq:RightAngledPentagon} (we refer to Figure 8), we have $$\sinh b_1 \sinh c_1' = \cosh a_{1}.$$
			And the inequality $c_1' \leq \ell_1(\gamma)/2 \leq \epsilon_0/2$ implies that
			\begin{equation*}
				\sinh b_1 = \cfrac{\cosh a_{1}}{\sinh c_1'} \geq \cfrac{1}{\sinh (\epsilon_0/2)} \ .
			\end{equation*}
			Therefore $b_1 \geq \text{arsinh}((\sinh (\epsilon_0 /2))^{-1})$.

			Let $M_0' = \max \{ 4/\ln(1+\sqrt{2}) , \text{arsinh} ((\sinh (\epsilon_0 / 2))^{-1})\}$. Then we have $b_1 \geq M_0'$. The
same argument implies $b_2 \geq M_0'$.
Thus we have
			\begin{equation}\label{eq:C2_Pb_2nd_2}
				b_i \geq M_0',\  i=1,\ 2\ .
			\end{equation}
			
			Since $b_i+d_i \geq M_0'$ and $c_i' < \epsilon_0/2$, $i=1,2$, we have $k_2>0$ (depending on $\epsilon_0$) such that
			\begin{equation*}\label{eq:C2_Pb_3rd_2}
				k_2^{-1} \cdot e^{b_i+d_i} \leq \sinh (b_i+d_i) \leq \cfrac{1}{2} \cdot e^{b_i+d_i}, \ i=1,2,
			\end{equation*}
			and
			\begin{equation*}\label{eq:C2_Pb_4th_2}
				c_i' \leq \sinh c_i' \leq k_2 \cdot c_i' \  ,\  i=1,2.
			\end{equation*}
		
			Similarly to the proof of \eqref{eq:C2_Pb_5th}, we will show that the difference between $b_i$ and $a_{i}$, $i=1,2$, is bounded from above.
We first study the difference between $b_1$ and $a_{1}$. There are two cases depending on whether $\ell_1(\gamma) < e^{-1} \ln(1+\sqrt{2})$ or not.

		\bigskip

			\begin{case}{(a): $\ell_1(\gamma) < e^{-1} \ln(1+\sqrt{2})$. }

\bigskip
				For this case, by the same argument as in the proof of \eqref{eq:C2_Pb_5th}, we have $|b_1-a_{1}|<D_2$.
			\end{case}

\bigskip
			\begin{case}{(b): $\ell_1(\gamma) \geq e^{-1} \ln(1+\sqrt{2})$.}

\bigskip
				In this case, since $d_1 = 0$, as in the proof of \eqref{eq:C2_Pb_5th}, we have
				\begin{equation*}
					e^{a_{1}} \geq k_2^{-1} c_1' e^{b_1}
                    \geq \cfrac{k_2^{-1}}{4} \ell_1(\gamma) e^{b_1}
                    \geq \cfrac{k_2^{-1} \ln(1+\sqrt{2})}{4 e} \cdot e^{b_1}
				\end{equation*}
				and
				\begin{equation*}
					e^{a_{1}} \leq k_2 c_1' e^{b_1}
                    \leq \cfrac{k_2 \epsilon_0}{2} \cdot e^{b_1} \ .
				\end{equation*}
				Let $D_2'=\max\{| \ln (k_2^{-1} \ln(1+\sqrt{2})) - \ln (4e)|,\ |\ln(k_2 \epsilon_0) - \ln 2|\}$.
				We have $|b_1-a_{1}| < D_2'$.
			\end{case}

The same proof applies to $a_2-b_2$ .
			Thus  we have a positive constant $D_2'$ depending on $\epsilon_0$ such that
			\begin{equation}\label{eq:C2_Pb_5th_2}
				|b_i-a_{i}|<D_2' \text{ , for } i=1,2.
			\end{equation}

The rest of the proof of  this lemma is identical to the proof of Lemma \ref{lemma:C2_Pb_Final} after the inequality \eqref{eq:C2_Pb_5th}. We omit the details.
\end{proof}

		The next lemma is the discussion for the $d_1/(b_2+d_2)$ part.
		
		\begin{lemma}\label{lemma:C2_Pd_Final_2}
			There exists a positive constant $K_4''$ depending on $\epsilon_0$ such that
			\begin{equation}\label{eq:C2_Pd_Final_2}
				\cfrac{d_1}{b_2+d_2} \leq K_4'' \cdot \max \{ 1, \ \cfrac{\ell_2(\gamma)}{\ell_1(\gamma)} \}.
			\end{equation}
		\end{lemma}
		
		\begin{proof}	We need to consider the two cases depending on whether $\ell_1(\gamma) < e^{-1} \ln(1+\sqrt{2})$ or not.
		
\bigskip
			\begin{case}{(a): $\ell_1(\gamma) < e^{-1} \ln(1+\sqrt{2})$. }
\bigskip

If, moreover,  $\ell_2(\gamma) < e^{-1} \ln(1+\sqrt{2})$, then
					by Lemma \ref{lemma:C2_Pd_Final}, we have
					\begin{equation}\label{eq:5-2}
						\cfrac{d_1}{b_2+d_2} \leq \cfrac{d_1}{d_2}
						\leq 2 \max\{1,\ \cfrac{\ell_2(\gamma)}{\ell_1(\gamma)}\}.
					\end{equation}
				Otherwise, $\ell_2(\gamma) \geq e^{-1} \ln(1+\sqrt{2})$.
					As $d_2 = 0$,  by \eqref{eq:4-2}, we have
					\begin{equation*}
						\begin{split}
							\cfrac{d_1}{b_2+d_2} = \cfrac{d_1}{b_2}
                            \leq \cfrac{2}{M_0} \ln \cfrac{\ell_1(C)}{\ell_1(\gamma)}.
						\end{split}
					\end{equation*}
                    Since $\ell_1(C)\geq \ell_1(\gamma)$, we have $$\ln \cfrac{\ell_1(C)}{\ell_1(\gamma)} \leq  \cfrac{\ell_1(C)}{\ell_1(\gamma)}.$$
                    It follows that
                    \begin{equation*}
                        \begin{split}
                            \cfrac{d_1}{b_2+d_2}
                            &\leq \cfrac{2}{M''} \cfrac{\ell_1(C)}{\ell_1(\gamma)}
                            = \cfrac{2  \ln(1+\sqrt{2})}{M'' \ell_2(\gamma)} \cfrac{\ell_2(\gamma)}{\ell_1(\gamma)}
                            =\cfrac{2 e }{M''} \cdot \cfrac{\ell_2(\gamma)}{\ell_1(\gamma)}.
                        \end{split}
                    \end{equation*}

			\end{case}

			\bigskip

			\begin{case}{(b): $\ell_1(\gamma) \geq e^{-1} \ln(1+\sqrt{2})$.}

\bigskip
				By assumption, $d_1=0$.
				It follows that $$d_1/(b_2+d_2) =0.$$
			\end{case}

Let $K_4'' = \max \{2, 2e / M_0'\}$. We are done.
		\end{proof}
		
		\begin{proposition}\label{lemma:Same_Boundary_2}
			Let $X_1,X_2$ be any hyperbolic metrics in the
$\epsilon_0$-relative part of $\mathcal{T}(S)$.
            For  any essential arc $\beta \in \mathcal{B}(S)$ with endpoints lying on the same boundary component $\gamma$ of $S$,
            let $\alpha$ and $\alpha'$ be the associated simple closed curves homotopic the  boundaries of a regular  neighborhood of
            $\beta \cup \gamma$.
			Then there exists a positive constant $K_4$ depending on $\epsilon_0$ such that

			\begin{equation}\label{eq:Same_Boundary_2}
				\begin{split}
					\cfrac{\ell_{X_1}(\beta)}{\ell_{X_2}(\beta)}
					&\leq K_4 \cdot \max \{ 1, \
					\cfrac{\ell_{X_1}(\alpha)}{\ell_{X_2}(\alpha)}, \
\cfrac{\ell_{X_1}(\alpha')}{\ell_{X_2}(\alpha')}, \
					\cfrac{\ell_{X_2}(\gamma)}{\ell_{X_1}(\gamma)} \}.
				\end{split}
			\end{equation}
		\end{proposition}
		
		\begin{proof}
			This is a corollary of  \eqref{eq:C2_Pb_Final_2} and \eqref{eq:C2_Pd_Final_2}.
		\end{proof}

\subsection{Conclusion}	
By Proposition \ref{lemma:Diff_Boundary_2} and Proposition \ref{lemma:Same_Boundary_2}, the same proof as that of Corollary \ref{lemma:key_1}
proves Theorem \ref{lemma:key}.
Theorem \ref{lemma:key} implies Theorem \ref{lemma:first}.

Recall that in the proof of Theorem \ref{lemma:first}, we just use elementary hyperbolic geometry.
However, for example, the inequality $(\ref{eq:Diff_Boundary})$ we have shown is not obvious.
Note that we only assume that the hyperbolic surfaces belong to the $\epsilon_0$-relative part of $\mathcal{T}(S)$
(not the thick part), thus the geodesic arcs we consider may cross
 long narrow cylinders or twist a lot.
 Thus it is difficult to control the lengths of arcs by that of simple closed curves.
What we have done is to show that the ratios can be controlled uniformly.

\begin{example}\label{example:Nielsen}
We use the Nielsen extension  of Riemann surfaces with boundary \cite{Bers, Papado-Th2010}
to  show that inequality \eqref{eq:couter} fails on
the whole $\epsilon_0$-relative part of $\mathcal{T}(S)$.

Let $X_0$ be any given hyperbolic metric on $S$. We can add each geodesic boundary component of
$X_0$  with an infinite funnel such that $X_0$ becomes the convex core of a Riemann surface
$R=\mathbb{H}^2/\Gamma$, where $\Gamma$ is a Fuchsian group of the second kind. By taking the double
of $R$, we obtain a Riemann surface $R^d$ without boundary. There is a unique hyperbolic metric in
the conformal class of $R^d$ and its restriction on $R$ defines a new hyperbolic metric $X_1$ on $S$
with geodesic boundary. We call $X_1$ the \emph{Nielsen extension} of $X_0$.

We may view $X_0$ as a conformal embedded subsurface of $X_1$. By the Schwarz Lemma, the Nielsen extension
decreases the hyperbolic metric on $X_0$. As a result, we have

$$\sup_{\alpha\in \mathcal{C}(S)}\{\frac{\ell_{X_1}(\alpha)}{\ell_{X_0}(\alpha)}\}\leq 1.$$

A theorem of Halpern \cite{H} shows that,
if $\alpha$ is a boundary curve of $X_0$ with length $l$, then the length of the corresponding boundary curve of the Nielsen extension
$X_1$ is less than $\frac{l}{2}$. In particular, if we define by $X_{n+1}$ be the Nielsen extension of $X_n$,
then $\ell_{X_n}(\alpha)\to 0$ for any boundary curve $\alpha$. Combined with the Collar Lemma,
we have
$$\sup_{\alpha\in \mathcal{C}(S)\cup \mathcal{B}(S)}\{\frac{\ell_{X_n}(\alpha)}{\ell_{X_0}(\alpha)}\}\to \infty$$
while
$$\sup_{\alpha\in \mathcal{C}(S)}\{\frac{\ell_{X_n}(\alpha)}{\ell_{X_0}(\alpha)}\}\leq 1.$$
Note that the sequence $(X_n)$ we constructed lies in some  $\epsilon_0$-relative part of $\mathcal{T}(S)$,
but not in any $\epsilon$-thick $\epsilon_0$-relative part.

On the other hand,  the ``$\epsilon_0$-relative" upper boundedness assumption on lengths of the boundary curves is necessary
for both inequality \eqref{eq:couter} and Theorem \ref{lemma:key}, see Example 3.8 in \cite{LPST}.
\end{example}

        \section{Applications and further study}
           \subsection{Moduli space.}
           Let $\text{Mod}(S)$ be the  modular group (or the mapping class group) of $S$.
           Recall that $\mathrm{Mod}(S)$ is the group of homotopy classes
            of orientation-preserving homeomorphism of $S$.
 $\text{Mod}(S)$ acts on the Teichm\"uller space $\mathcal{T}(S)$ by switching the markings.
           Moreover, the action is  properly
discontinuous and by isometries (here we endow $\mathcal{T}(S)$ with  the length spectrum metric or the arc-length spectrum metric).
            The moduli space of $S$, denoted by $\mathcal{M}(S)$, is the quotient space $$\mathcal{M}(S) = \mathcal{T}(S) / \text{Mod}(S).$$
            We have the natural projective map $\pi : \mathcal{T}(S) \to \mathcal{M}(S)$.

            For any fixed positive number $\epsilon_0$, the subset of $\mathcal{M}(S)$ consisting of hyperbolic structures with
             lengths of boundary components bounded above by $\epsilon_0$ is called the \emph{$\epsilon_0$-relative part} of $\mathcal{M}(S)$.

            The metric $d$ on $\mathcal{T}(S)$ induces a metric $d^{M}$ on $\mathcal{M}(S)$ by letting
            $$
            d^{M}(\tau_1, \tau_2) = \inf_{X_i \in \pi^{-1}(\tau_i)} d(X_1,X_2).
            $$
            Similarly, we have corresponding metrics $\bar{d}^{M}$, $\delta_L^{M}$ and $d_L^{M}$ on $\mathcal{M}(S)$ which are induced by $\bar{d}$, $\delta_L$ and $d_L$ on $\mathcal{T}(S)$, respectively.

            The following result is a direct corollary of Theorem \ref{lemma:first}.
            \begin{corollary}\label{cor:moduli_first}
                Given $\epsilon_0 > 0$.
                Let $d_L^{(M)}$ and $\delta_L^{(M)}$ be the length spectrum metric and the arc-length spectrum metric on $\mathcal{M}(S)$.
                For any $\tau_1$, $\tau_2$ in the $\epsilon_0$-relative part of $\mathcal{M}(S)$, we have
                \begin{equation}\label{eq:moduli_first}
                    d_L^{M}(\tau_1,\tau_2) \leq \delta_L^{M}(\tau_1,\tau_2) \leq d_L^{M}(\tau_1,\tau_2) + C,
                \end{equation}
                where $C$ is a positive constant depending on $\epsilon_0$.
            \end{corollary}

In the case where $S$ is a surface of finite type without boundary, the authors  \cite{AlmostModuli} proved that the length spectrum metric and the Teichm\"uller metric are almost isometric on the moduli space $\mathcal{M}(S)$.
The result can not be generalized to surfaces of finite type with boundary \cite{LSSZ}. However, we ask
the following

\begin{problem}
Let $S$ be a surface of finite type with boundary. Are the arc-length spectrum metric and the Teichm\"uller metric  almost isometric on the moduli space $\mathcal{M}(S)$?
\end{problem}

        \subsection{Metrics on Teichm\"uller spaces of surfaces of infinite type}\label{section:infinite}

            A surface is said to be of \emph{finite type} if its fundamental group is finitely generated.
            Otherwise it is said to be of \emph{infinite type}.
            For more details on  Teichm\"uller spaces of surfaces of infinite type (where the definition of Teichm\"uller space is not unique and more involved), we refer to  \cite{ALPSS}.

            A hyperbolic surface $S$ (possibly with geodesic boundary) is said to be \emph{convex} if for every pair of points $x, y \in S$ and for every arc $\gamma$ with endpoints $x$ and $y$, there exists a geodesic arc of $S$ connecting $x$ and $y$ that is homotopic to $\gamma$ relative to the endpoints.

            A convex hyperbolic surface $S$ with geodesic boundary is \emph{Nielsen convex} if every point of $S$ is contained in a geodesic arc with endpoints contained in some simple closed geodesics of $S$.
            For a hyperbolic surface of finite type to be Nielsen convex is equivalent to be convex with geodesic boundary and of finite area.
            However, for surfaces of infinite type the two notions  maybe not equivalent \cite{ALPSS}.

            In this following, we assume that $S$ is a hyperbolic surfaces of  infinite type and $S$ is Nielsen convex.
Moreover, we assume that all the boundary components of $S$ are of length less than some positive constant.

            Denote by $\mathcal{T}_L(S)$ the length spectrum Teichm\"uller space of $S$, which consists of
            (equivalence classes) of hyperbolic surfaces $X$ that are homeomorphic with $S$ and satisfy
            \begin{equation*}
                d_L (S, X) = \log \sup_{\alpha \in \mathcal{C}(S)} \{ \cfrac{\ell_X(\alpha)}{\ell_S(\alpha)}, \ \cfrac{\ell_S(\alpha)}{\ell_X(\alpha)} \}<\infty.
            \end{equation*}
            We endow $\mathcal{T}_L(S)$ with the length spectrum metric
            \begin{equation*}
                d_L (X,Y) = \log \sup_{\alpha \in \mathcal{C}(S)} \{ \cfrac{\ell_Y(\alpha)}{\ell_X(\alpha)}, \ \cfrac{\ell_X(\alpha)}{\ell_Y(\alpha)} \}.
            \end{equation*}

            We define the \emph{$\epsilon_0$-relative part} of $\mathcal{T}_L(S)$ to be the subset consisting of hyperbolic surfaces with lengths of boundary components bounded above by $\epsilon_0$. By assumption on $S$, the $\epsilon_0$-relative part of $\mathcal{T}_L(S)$ is not a empty set if $\epsilon_0$ is sufficiently large.
We can also define the arc-length spectrum metric by
            \begin{equation*}
                \delta_L (X, Y) = \log \sup_{\alpha \in \mathcal{C}(S) \bigcup \mathcal{B}(S)} \{ \cfrac{\ell_Y(\alpha)}{\ell_X(\alpha)}, \ \cfrac{\ell_X(\alpha)}{\ell_Y(\alpha)} \}.
            \end{equation*}

        As the discussions in the previous sections are not related to the topological type of surface, we have the following theorem.

            \begin{theorem}
                Given $\epsilon_0 > 0$.
                Let $d_L$ and $\delta_L$ be the length spectrum metric and the arc-length spectrum metric on $\mathcal{T}_L(S)$.
                For any $X$ and $Y$ in the $\epsilon_0$-relative part of $\mathcal{T}_L(S)$, we have
                \begin{equation*}
                    d_L (X, Y) \leq \delta_L (X, Y) \leq d_L (X, Y) + C,
                \end{equation*}
                where $C$ is a positive constant depending on $\epsilon_0$.
            \end{theorem}
\subsection{Further study.}\label{sec:fur} $\\$

Note that the constant $C=C(\epsilon_0)$ in Theorem \ref{lemma:first} only depend on $\epsilon_0$.
\begin{problem}
Does the constant $C=C(\epsilon_0)$ in Theorem \ref{lemma:first} tends to $0$ as $\epsilon_0$ tends to $0$?
\end{problem}

It was shown in \cite{LPST} that for surfaces of finite type with boundary,

\begin{equation}\label{eq:Thurston}
			\begin{split}
				\log \sup_{\alpha \in  \mathcal{B}(S)\cup \mathcal{C}(S)} \{ \cfrac{\ell_X(\alpha)}{\ell_Y(\alpha)}\}
				&= \log \sup_{\alpha \in  \mathcal{B}(S)\cup \partial S} \{ \cfrac{\ell_X(\alpha)}{\ell_Y(\alpha)} \}.
			\end{split}
		\end{equation}
for any $X,Y\in \mathcal{T}(S)$. This gives new formulae for Thurston's metric and the arc-length spectrum metric.
 The above equality $(\ref{eq:Thurston})$ was proved by using Thurston's theory of measured laminations.

 \begin{problem}
Does the equality $(\ref{eq:Thurston})$ hold on Teichm\"uller spaces of surfaces of infinite type?
\end{problem}


\begin{thebibliography}{CGQ}

\bibitem{Abikoff} Abikoff W.: The real analytic theory of Teichm\"uller space. Lecture Notes in Mathematics
820. Springer-Verlag (1980).

		\bibitem{ALPSS} Alessandrini D., L. Liu, A. Papadopoulos, W. Su, and Z. Sun: On Fenchel-Nielsen coordinates on Teichm\"uller spaces of surfaces of infinite type.  Ann. Acad. Sci. Fenn. Math. 36, no.2, 2011, 621-659.

\bibitem{ALPS2} Alessandrini D., L. Liu, A. Papadopoulos, W. Su: The horofunction compactification of the arc metric on Teichm\"uller space.  arXiv:1411.6208.


\bibitem{Bers} Bers L.: Nielsen extensions of Riemann surfaces. Ann. Acad. Sci. Fenn. 2, 1976, 29-34.

\bibitem{Buser} Buser P.: Geometry and spectra of compact Riemann surfaces. Reprint of the 1992 edition. Modern Birkh\"auser Classics. Birkh\"auser Boston, Inc., Boston, MA, 2010.


		\bibitem{Rafi} Choi Y. and K. Rafi: Comparison between Teichm\"uller and Lipschitz metrics. J. Lond. Math. Soc. (2) 76, no.3, 2007, 739-756.
\bibitem{DGK} Danciger J., F. Gu\'eritaud and F. Kassel: Margulis spacetimes via the arc complex. arXiv:1407.5422.
\bibitem{Earle} Earle C. J.: Reduced Teichm\"uller spaces. Transactions of the American Mathematical Society, 1967, 54-63.

\bibitem{ES} Earle C. J. and A. Schatz: Teichm\"uller theory for surfaces with boundary. Journal of Differential Geometry,  4(2), 1970, 169-185.


\bibitem{H} Halpern N.: Some remarks on Nielsen extensions of Riemann surfaces. The Michigan Mathematical Journal, 30(1), 1983, 65--68.

\bibitem{Li} Li Z.:  Teichm\"uller metric and length spectrums of Riemann surfaces. Sci. Sinica Ser. A 29, no. 3, 1986, 265--274.

\bibitem{LSW} Liu L., Z. Sun and H. Wei: Topological equivalence of metrics in Teichm\"uller space. Ann. Acad. Sci. Fenn. Math. 33, no. 1, 2008, 159--170.

\bibitem{Liu1999} Liu L.: On the length spectrums of non-compact Riemann surfaces. Ann. Acad. Sci. Fenn. Math. 24, 1999, 11--22.

\bibitem{Liu2001} Liu L.: On the metrics of length spectrum in Teichm\"uller space. Chinese J. Cont. math. 22 (1), 2001, 23--34.



		\bibitem{AlmostModuli} Liu L. and W. Su:  Almost-isometry between Teichm\"uller metric and length-spectrum metric on moduli space. Bull. Lond. Math. Soc. 43, no.6, 2011, 1181-1190.

		\bibitem{LPST} Liu L., A. Papadopoulos, W. Su and G. Th\'eret:  On length spectrum metrics and weak metrics on Teichm\"uller spaces of surfaces with boundary. Ann. Acad. Sci. Fenn. Math. 35, no.1, 2010, 255-274.

		\bibitem{LTBoundary} Liu L., A. Papadopoulos, W. Su and G. Th\'eret: Length spectra and the Teichm\"uller metric for surfaces with boundary. Monatsh. Math. 161, no.3, 2010, 295-311.

\bibitem{LSSZ} Liu L., H. Shiga, W. Su and Y. Zhong, in preparation.
\bibitem{Marden} Marden A.: Outer circles: An introduction to hyperbolic 3-manifolds. Cambridge University Press, 2007.	

\bibitem{Minsky} Minsky Y.: Extremal length estimates and product regions in Teichm\"{u}ller space. Duke Math.J. 83, 1996, 249-286.


\bibitem{Papado-Th2007} Papadopoulos T. and G. Th\'eret:  On the topology defined by Thurston's asymmetric metric.  Math. Proc. Camb. Phil. Soc., 142, 2007, 487-496.


\bibitem{Papado-Th2010} Papadopoulos A. and G. Th\'eret: Shortening all the simple closed geodesics on surfaces with boundary. Proc. Amer. Math. Soc. 138, no. 5, 2010, 1775-1784.




\bibitem{Shiga} Shiga H.: On a distance defined by the length spectrum of Teichm\"uller space. Ann. Acad. Sci. Fenn. Math. 28, no. 2, 2003, 315--326.


\bibitem{S} Sorvali T.:  The boundary mapping induced by an isomorphism of covering groups. Ann. Acad. Sci. Fenn. Ser. A I No. 526 (1972).

\bibitem{Sorvali} Sorvali T.:  On Teichm\"uller spaces of tori. Ann. Acad. Sci. Fenn. Ser. A I Math. 1, 1975, 7-11.
	

 \bibitem{Thurston-notes} Thurston W. P.: The geometry and topology of Three-manifolds. Mimeographed notes, Princeton University, 1976.


\bibitem{Thurston1998} Thurston W. P.:  Minimal stretch maps between hyperbolic surfaces. preprint, 1986, Arxiv:math GT/9801039.

\bibitem{Wolpert} Wolpert S.: The length spectra as moduli for compact Riemann surfaces. Ann. of Math. 109, 1979, 323-351.























		\end{thebibliography}
\end{document}